\documentclass[a4paper,reqno,12pt,draft]{amsart}
\usepackage{amsmath}
\usepackage{amssymb}
\usepackage{amsfonts}
\usepackage{amsthm}
\usepackage{enumerate}
\usepackage{comment}
\usepackage{cite}
\usepackage{appendix}
\usepackage[dvips]{graphicx,color}
\usepackage{delarray}
\usepackage{ascmac}
\usepackage{fancybox}
\usepackage{lineno}
\usepackage[hidelinks,draft=false]{hyperref}
\theoremstyle{plain}
\newtheorem{theorem}{Theorem}[section]
\newtheorem{corollary}[theorem]{Corollary}

\newtheorem{proposition}[theorem]{Proposition}
\newtheorem{lemma}[theorem]{Lemma}
\newtheorem{remark}[theorem]{Remark}

\numberwithin{equation}{section}

\numberwithin{equation}{section}

\def\XXint#1#2#3{{\setbox0=\hbox{$#1{#2#3}{\int}$}
\vcenter{\hbox{$#2#3$}}\kern-.5\wd0}}

\begin{document}
\title[Classification of bifurcation structure]{Classification of bifurcation structure for semilinear elliptic equations in a ball}
\author{Kenta Kumagai}
\address{Department of Mathematics, Institute of Science Tokyo}
\thanks{This work was supported by JSPS KAKENHI Grant Number 23KJ0949}
\email{kumagai.k.ah@m.titech.ac.jp}
\date{\today}

\begin{abstract}
We consider the Gelfand problem with Sobolev supercritical nonlinearities $f$ in the unit ball. In the case where $f$ is a power type nonlinearity or the exponential nonlinearity, it is well-known that the bifurcation curve has infinitely many turning points when the growth rate of $f$ is smaller than that of the specific nonlinearity (called the Joseph-Lundgren critical nonlinearity), while the bifurcation curve has no turning point when the growth rate of $f$ is greater than or equal to that of the Joseph-Lundgren critical nonlinearity. 

In this paper, we give a new type of nonlinearity $f$ such that the growth rate is greater than or equal to that of the Joseph-Lundgren critical nonlinearity, while the bifurcation curve has infinitely many turning points. This result shows that the bifurcation structure is not determined solely by the comparison between the growth rate of $f$ and that of the Joseph-Lundgren critical nonlinearity. In fact, we 
find a general criterion which determines the bifurcation structure; and give a classification of the bifurcation structure.
\end{abstract}

\maketitle

\raggedbottom

\section{Introduction}
Let $N\ge 3$ and $B_1\subset \mathbb{R}^N$  
be the unit ball. We are interested in the global bifurcation diagram of the elliptic equation
\begin{equation}
\label{gelfand}
-\Delta u=\lambda f(u) \hspace{2mm} \text{in } B_1, \hspace{4mm}
u>0 \hspace{2mm}\text{in } B_1, \hspace{4mm}
u=0  \hspace{2mm}\text{on } \partial B_1,
\end{equation}
where $\lambda>0$ is a parameter. Throughout this paper, we suppose without further mentioning the following condition on $f$:
\begin{equation}
\label{asf0}
    \text{$f\in C^2[0,\infty)$}\hspace{4mm}\text{and}\hspace{4mm}\text{$f>0$ in $[0,\infty)$}.
\tag{$f_1$}
\end{equation}
By the symmetric result of Gidas, Ni and Nirenberg \cite{Gidas}, every solution of \eqref{gelfand} is radially symmetric and decreasing. Moreover, it is known by \cite{korman, Mi2015} that 
the set of solutions of \eqref{gelfand} is an unbounded $C^2$ curve 
described by $\{(\lambda(\alpha), u(r, \alpha)) ;\alpha>0\}$, where $\lVert u(r, \alpha) \rVert_{L^{\infty}(B_1)}=u(0,\alpha)=\alpha$ and $\lambda(\alpha)\to 0$ as $\alpha\to 0$. 

A celebrated result of Joseph and Lundgren \cite{JL} states that the growth rate of $f$ and the dimension $N$ make a change to the bifurcation structure of \eqref{gelfand} when $f=(1+u)^p$ or $f=e^u$. More precisely, for the exponential nonlinearity, they showed that the bifurcation curve oscillates around some $\lambda^{*}>0$ and $\lambda(\alpha)\to \lambda^{*}$ as $\alpha\to\infty$ if $3\le N\le 9$. We call the property of the curve Type I. On the contrary, $\lambda(\alpha)$ is monotonically increasing and $\lambda(\alpha)\to\lambda^{*}$ as $\alpha\to\infty$ with some $\lambda^{*}>0$ if $N\ge 10$. We call the property of the curve Type II. For the case $f=(1+u)^p$, they showed that the bifurcation diagram is of Type I if $\frac{N+2}{N-2}<p<p_{\mathrm{JL}}$ and the bifurcation diagram is of Type II if $p\ge p_{\mathrm{JL}}$, where
\begin{equation*}
p_{\mathrm{JL}}=
\begin{cases}
    1+\frac{4}{N-4-2\sqrt{N-1}}&\text{if $N\ge 11$,}\\
    \infty & \text{if $N\le 10$}.
\end{cases}
\end{equation*}
In addition to these types, Miyamoto \cite{Mi2014} proved that when $N\ge 11$ and
$f=(u+\varepsilon)+(u+\varepsilon)^p$ with large $p>p_{\mathrm{JL}}$ and small $\varepsilon>0$, the bifurcation curve exhibits another type such that 
the curve turns finitely many times and $\lambda(\alpha)\to \lambda^{*}$ with some $\lambda^{*}>0$. We call the property of the curve Type III. Here, we remark that bifurcations of Type III can be regarded as variants of bifurcations of Type II. Indeed, if the bifurcation diagram is of Type III with some analytic function $f$ on $[0,\infty)$ satisfying \eqref{asf1} below, then there exists $c_0>0$ so that the bifurcation diagram is of Type II for $f_c(u):=f(u+c)$ with $c>c_0$ (see Proposition \ref{apenprop}).

Motivated from the studies stated above, we arrive at the fundamental question: when the growth rate of $f$ is greater than or equal to that of the Joseph-Lundgren critical nonlinearity (i.e, $f=e^u$ with $N=10$ or $f=(1+u)^{p_{\mathrm{JL}}}$ with $N\ge 11$), does the bifurcation curve always exhibit Type II or Type III? The following result gives a negative answer to this question. 
\begin{theorem}
\label{main1}
Let $N\ge 10$. We assume that $\gamma\in \mathbb{R}$ satisfies 
$\gamma>-\log 2$ if $N\ge 11$.
We define
\begin{equation*}
    f_{N,\gamma}(u):=
    \begin{cases}
        e^u(1+u)^{\gamma} &\text{if $N=10$,}\\
    (1+u)^{p_{\mathrm{JL}}}\left(1+\frac{\gamma}{\log (2+u)}\right)  &\text{if $N\ge 11$.}
    \end{cases}
\end{equation*}
Then, the bifurcation diagram of \eqref{gelfand} is of 
\begin{enumerate}
    \item[{(i)}] Type I if $\gamma>\frac{1}{4\sqrt{N-1}}$;
    \item[{(ii)}] Type II or Type III if $\gamma<\frac{1}{4\sqrt{N-1}}$.
\end{enumerate}
In addition, in the case $\gamma<\frac{1}{4\sqrt{N-1}}$, there exists $c_0\ge 0$ such that the bifurcation diagram of \eqref{gelfand} is of Type II when $f(u)=f_{N,\gamma}(u+c)$ and $c>c_0$. Moreover, in the case $\gamma>\frac{1}{4\sqrt{N-1}}$, the bifurcation diagram of \eqref{gelfand} is always of Type I when $f(u)=f_{N,\gamma}(u+c)$ with $c\in\mathbb{R}$ satisfies $f(0)>0$.
\end{theorem}
The bifurcation structure of \eqref{gelfand} has been studied for various Sobolev supercritical nonlinearities $f$ such as $f=u^{p}+o(u^{p})$ with $p>\frac{N+2}{N-2}$, $f=e^u+o(e^u)$,
$f=e^{u^p}$ with $p>0$, $f=e^{e^u}$ and so on.
We refer to \cite{Dan,KiWei,Dancer-e, Marius,Mi2014,Mi2015} for \eqref{gelfand} and \cite{Nor,Flore,Guowei,chend} for related problems. As a result, Miyamoto \cite{Mi2018}, and Miyamoto and Naito \cite{Mi2024} unified these works by considering the following assumption
\begin{equation}
\label{asf1}
F(u):=\int_{u}^{\infty}\frac{ds}{f(s)}<\infty \hspace{4mm}\text{and}\hspace{4mm}q:=\lim_{u\to\infty}\frac{f'^2(u)}{f(u)f''(u)}<\frac{N+2}{4}.
\tag{$f_2$}
\end{equation}
If \eqref{asf0} and \eqref{asf1} hold, we deduce from Lemma \ref{convexlem} and L'Hospital's rule that $q\ge 1$, $f'(u), f''(u)>0$ for $u>u_0$ with some $u_0>0$, and 
\begin{equation}
\label{sim1}
   \lim_{u\to\infty} F(u)f'(u)=\lim_{u\to\infty} \frac{(F(u))'}{(1/f'(u))'}=q. 
\end{equation}
Moreover, we mention that the H\"older conjugate of the ratio $q$ measures a variant of the growth rate of $f$. Indeed, we can easily confirm that $q=\frac{p}{p-1}$ if $f=(1+u)^p$ and $q=1$ if $f=e^{u}$. In addition, by Lemma \ref{convexlem} and L'Hospital's rule, we obtain  
\begin{equation*}
    \lim_{u\to\infty}\frac{f(u)}{f'(u)u}=\lim_{u\to\infty}\frac{f(u)/f'(u)}{u}=\lim_{u\to\infty}\left(1-\frac{f(u) f''(u)}{f'^2(u)}\right)=\frac{q-1}{q};
\end{equation*}
\begin{equation*}
    \lim_{u\to\infty}\frac{\log f(u)}{\log f'(u)}=\lim_{u\to\infty}\frac{f'(u)/f(u)}{f''(u)/f'(u)}=\lim_{u\to\infty}\frac{f'^2(u)}{f(u)f''(u)}=q.
\end{equation*}
On the other hand, we remark that the finer growth rate cannot be measured by using $q$. Indeed, we can confirm that $q=1$ in the case $f=e^{u^p}$ or $f=e^{e^u}$ with $p>0$, and $q=\frac{p}{p-1}$ in the case $f=(1+u)^{p}(\log (2+u))^\tau$ with $\tau>0$. In particular, when $f=f_{N, \gamma}$ defined in Theorem \ref{main1}, we have $q=q_{\mathrm{JL}}$ for any $\gamma\in \mathbb{R}$. 

By focusing on the ratio $q$, Miyamoto \cite{Mi2018} proved that the bifurcation diagram is always of Type I when $q_{\mathrm{JL}}<q$, where 
\begin{equation}
\label{qjl}
    q_{\mathrm{JL}}=\frac{N-2\sqrt{N-1}}{4}.
\end{equation}
In addition, when $q<q_{\mathrm{JL}}$, Miyamoto and Naito \cite{Mi2024} proved that there exists $c_0>0$ such that the bifurcation diagram is of Type II with $f_c=f(u+c)$ for any $c>c_0$. These results imply that the bifurcation structure is solely determined by the ratio $q$ (except for the difference with respect to the translation of $f$) in the case $q\neq q_{\mathrm{JL}}$; and the question stated above is true for the case $q\neq q_{\mathrm{JL}}$.

On the other hand, for the case $q=q_\mathrm{JL}$, there are few results in the literature clarifying the bifurcation structure of \eqref{gelfand}. For the special case where $f=e^u-1$ or $f=8\exp[2e^u-u-2]+\eta\exp[2e^u-2u-2]$ with $0\le \eta\le 4$, the bifurcation structure is clarified by \cite{chend} and \cite{Mi2024}, respectively. However, to the best of the author's knowledge, the bifurcation structure has not been clarified except for the above special nonlinearities. 

Theorem \ref{main1} tells us that in contrast to the case $q\neq q_{\mathrm{JL}}$, the bifurcation curves exhibit essentially different types depending on the perturbation of $f$ for the case $q=q_{\mathrm{JL}}$. Moreover, Theorem \ref{main1} shows that the bifurcation structure is not determined by the comparison of the growth rate of $f$ and that of the Joseph-Lundgren critical nonlinearity for the case $q=q_{\mathrm{JL}}$. 

The main result stated below shows that the following quantity
\begin{equation}
\label{asF1}
(F(u)f'(u)-q_{\mathrm{JL}})(-\log F(u))^{k}\to \gamma \hspace{4mm}\text{as $u\to\infty$}
\tag{$F_1$}
\end{equation}
is a key factor in determining the bifurcation structure. Moreover, the result
contains Theorem \ref{main1} as a special case.


\begin{theorem}
\label{mainthms}
Let $N\ge 10$, $\gamma\in \mathbb{R}$ and $0<k\le 2$. We assume \eqref{asf0}, \eqref{asf1} and \eqref{asF1},
where $q_\mathrm{JL}$ is that in \eqref{qjl}. Then, the bifurcation diagram of \eqref{gelfand} is of 
\begin{itemize}
    \item[{(i)}] Type I if 
    \begin{equation}
       \text{$0<k<2$ and $\gamma>0$, \hspace{2mm} or \hspace{2mm} $k=2$ and $\gamma>\frac{1}{4\sqrt{N-1}}$;} 
    \tag{A}
    \end{equation}
    \item[{(ii)}] Type II or Type III if $f$ can be extended analytically on $(-\delta,\infty)$ with some $\delta>0$ and
    \begin{equation}
        \text{$0<k<2$ and $\gamma<0$, \hspace{2mm} or \hspace{2mm} $k=2$ and $\gamma<\frac{1}{4\sqrt{N-1}}$.}
        \tag{B}
    \end{equation}
\end{itemize}
In addition, when $(B)$ is satisfied, there exists $c_0\ge 0$ such that the bifurcation diagram of \eqref{gelfand} is of Type II with $f_{c}(u)=f(u+c)$ for all $c>c_0$. Moreover, the bifurcation curve does not turn if $\alpha:=\lVert u \rVert_{L^{\infty}(B_1)}$ is sufficiently large.
\end{theorem}
\begin{remark}
\label{Hito}
\rm{Since $F(u)\to 0$ as $u\to\infty$ and \eqref{sim1} is satisfied,
we have $q=q_{\mathrm{JL}}$ under the assumptions of Theorem \ref{mainthms}, where $q$ is that in \eqref{asf1}.  
Moreover, we remark that the analyticity of $f$ on $(-\delta,\infty)$ is used only to ensure that the critical points of $\lambda(\alpha)$ are finite on any bounded interval. The analyticity condition is used widely in the field. We refer to \cite{Mi2014, Dan, Figalli, kuma2025, kata}} for instance.
\end{remark}
As mentioned above, the bifurcation structure is essentially determined only by $q$ when $q\neq q_{\mathrm{JL}}$. Theorem \ref{mainthms} shows that the bifurcation structure is essentially determined not only by $q$, but also by the asymptotic behavior of $Ff'-q$ when $q=q_{\mathrm{JL}}$. We also mention that the bifurcations of Type I have not been confirmed in the literature for the case $q=q_{\mathrm{JL}}$ to the best of the author's knowledge. Theorem \ref{mainthms} shows that the bifurcations of Type I are common phenomena even for the case $q=q_{\mathrm{JL}}$.
In fact, as a corollary of Theorem \ref{mainthms}, we obtain
\begin{corollary}
\label{maincor}
The bifurcation diagram of \eqref{gelfand} is of
\begin{itemize}
\item[{(i)}] Type I if $N=10$ and
    $f$ is one of the following 
    \begin{equation*}
\text{$f=e^{u+\beta (1+u)^{\theta}}$,\hspace{2mm} $f=e^{(1+u)^{\theta}}$,\hspace{2mm} $f=e^{u}(1+u)^{\beta+\frac{1}{12}}$\hspace{2mm} or \hspace{2mm} $f=e^{(\log (1+u))^{\tau}}$}
\end{equation*}
with some $0<\theta<1$, $\beta>0$ and $\tau>1$.
\item[{(ii)}] Type II or Type III if $N=10$ and
$f$ is one of the following
\begin{equation*}
\text{$f=e^{u+\beta (1+u)^{\theta}}$, \hspace{2mm}$f=e^{(1+u)^{\tau}}$,\hspace{2mm} $f=e^{u}(1+u)^{\beta +\frac{1}{12}}$ or \hspace{2mm} $f=e^{e^u}$}
\end{equation*}
with $0<\theta<1$, $\beta<0$ and $\tau>1$.
\item[{(iii)}] Type I if $f=(1+u)^{p_{\mathrm{JL}}}\zeta(u)$ with $N\ge 11$ and $\zeta$ is one of the following
\begin{equation*}
\text{$\zeta=\frac{\alpha+\log (u+2)}{\log (u+2)}$, $\zeta=\frac{\beta+(\log (u+2))^{\theta}}{(\log (u+2))^{\theta}}$, $\zeta=e^{\tau (\log (u+1))^{\theta}}$ or $\zeta=(\log (2+u))^{\tau}$}
\end{equation*}
with $\alpha>\frac{1}{4\sqrt{N-1}}$, $0<\theta<1$, $\beta>0$ and $\tau<0$.
\item[(iv)] Type II or Type III if $f=(1+u)^{p_{\mathrm{JL}}}\zeta(u)$ with $N\ge 11$ and $\zeta$ is one of 
\begin{equation*}
\text{$\zeta=\frac{\alpha+\log (u+2)}{\log (u+2)}$, $\zeta=\frac{\beta+(\log (u+2))^\theta}{(\log (u+2))^\theta}$,  $\zeta=e^{\tau(\log (u+1))^\theta}$ or $\zeta=(\log (2+u))^{\tau}$}
\end{equation*}
with $\alpha<\frac{1}{4\sqrt{N-1}}$, $0<\theta<1$, $\beta<0$ and $\tau>0$, and $\zeta$ satisfies $\zeta(0)>0$.
\end{itemize}
\end{corollary}
The key of the proof of Theorem \ref{mainthms} is to obtain the following 
theorem on the asymptotic behavior and the Morse index of the radial singular solution.
\begin{theorem}
\label{singularthm}
We assume \eqref{asf0}, \eqref{asf1} and \eqref{asF1} with $0<k\le 2$ and $\gamma\in \mathbb{R}$. Then, the equation \eqref{gelfand} has the unique radial singular solution $(\lambda_{*},U_{*})$. Moreover, we have $U_{*}\in H^{1}_{0}(B_1)$, 
\begin{equation*}
    \lambda(\alpha)\to\lambda_{*}, \hspace{4mm} \text{and}\hspace{4mm} u(r,\alpha)\to U_{*} \hspace{2mm}\text{in $C^2_{\mathrm{loc}}(0,1]$} \hspace{4mm}\text{as $\alpha\to\infty$},
\end{equation*}
\begin{equation*}
    \lambda_{*} f'(U_{*})=\frac{(N-2)^2}{4r^2}+\frac{\gamma\sqrt{N-1}+o(1)}{2^{k-2}r^2(-\log r)^k}\hspace{4mm}\text{as $r\to 0$},
\end{equation*}
and 
\begin{itemize}
    \item[{(i)}] $m(U_{*})=\infty$ if (A) is satisfied;
    \item[{(ii)}] $m(U_{*})<\infty$ if (B) is satisfied,
\end{itemize}
where $(A)$ and $(B)$ are those in Theorem \ref{mainthms}.
\end{theorem}
Here, we say that $(\lambda_{*}, U_{*})$ is a radial singular solution of \eqref{gelfand} if $U_{*}\in C^2(0,1]$ is a regular solution of \eqref{gelfand} in $(0,1]$ with $\lambda=\lambda_{*}$ so that $U_{*}(r)\to \infty$ as $r\to 0$. Moreover, the Morse index $m(u)$ is defined by the maximal dimension of a subspace $X_k\subset C^{1}_{\mathrm{c}}(B_1)$ satisfying 
\begin{equation*}
    Q_{u}(\xi):=\int_{B_1}\left(|\nabla\xi|^2-\lambda f'(u)\xi^2\right)\,dx<0 \hspace{4mm}\text{for any $\xi\in X_k \setminus \{0\}$}.
\end{equation*}
We mention that the Morse induces of radial singular solutions play a key role in the bifurcation structure. Indeed, Brezis and V\'azquez \cite{Br} showed that the bifurcation diagram is of Type II if and only if $m(U_{*})=0$ for the case where $U_{*}\in H^1_{0}(B_1)$ and $f$ is non-decreasing and convex. In addition, Miyamoto \cite{Mi2014} expected that $m(U_{*})$ is equal to the number of turning points. Thus, various properties for radial singular solutions have been studied for many authors (see \cite{Luo,Merle,chen,Guowei,chend,Mi2014,Mi2015,Mi2018,Mi2020,KiWei,Marius}), and finally Miyamoto and Naito \cite{Mi2023} showed the uniqueness of the radial singular solutions and 
\begin{equation}
\label{r-as}
    \lambda_{*} f'(U_{*})=\frac{q(2N-4q)}{r^2}(1+o(1)) \hspace{4mm}\text{as $r\to0$}
\end{equation}
for all $f$ satisfying \eqref{asf1}. We can deduce from \eqref{r-as} and Hardy's inequality that $m(U_{*})=\infty$ for $q>q_{\mathrm{JL}}$ and $m(U_{*})<\infty$ for $q<q_{\mathrm{JL}}$. On the contrary, in the case $q=q_{\mathrm{JL}}$, \eqref{r-as} is not sufficient to study the Morse index of $U_{*}$ since $q_{\mathrm{JL}}(2N-4q_{\mathrm{JL}})$ is exactly the best constant $\frac{(N-2)^2}{4}$ of Hardy's inequality. Theorem \ref{singularthm} gives a refined asymptotic profile of $U_{*}$ than that in \eqref{r-as}, which enables us to study the Morse index of the radial singular solution and clarify the bifurcation structure.

The paper is organized as follows. In Section 2, we give examples of $f$ satisfying \eqref{asf1} and \eqref{asF1}. In Section 3, we study the asymptotic behavior and the Morse index of the radial singular solution, by applying a specific transformation used in \cite{Mi2023} and the fixed point theorem. As a result, we prove Theorem \ref{singularthm}. In Section 4, we prove separation/intersection results, which play key roles in studying the bifurcation structure. In Section 5, we prove Theorems \ref{main1}, \ref{mainthms} and Corollary \ref{maincor}. We mention that Proposition \ref{mainprop} obtained by using the idea in \cite{NS,kuma2024} plays an important role to obtain Theorem \ref{mainthms}.
\section{Preliminaries}
We first introduce the following
\begin{lemma}
\label{convexlem}
Assume that $f$ satisfies \eqref{asf0} and \eqref{asf1}. Then, we have $q\ge 1$. In addition, there exist $C_0>0$, $C_1>0$, $u_0>1$ and $\frac{N+2}{N-2}<p_0$ such that
\begin{itemize} 
    \item[{(i)}] $(p_0+1)\int_{0}^{u}f(t)\,dt<uf(u)$ for $u>u_0$;
    \item[{(ii)}] $u^{p_{0}-1}\le C_0 f'(u)$ for $u>u_0$;
    \item[{(iii)}] $C_{1}^{-1}u^{p_0-2}<f''(u)$ for $u>u_0$.
\end{itemize}
In particular, we obtain $f'(u)\to\infty$ as $u\to\infty$. In addition, if $N\ge 10$, we can take $p_0$ so that $\frac{N+2}{N-2}<p_0<2$.
\end{lemma}
\begin{proof}
Thanks to \cite[Lemma 2.1]{Mi2023}, we have $q\ge 1$ and $f'(u)\to\infty$ as $u\to\infty$. 

Now, we prove (i) by using the argument used in the proof of \cite[Theorem 1.1]{Mi2023}. By L'Hospital's rule, we have
\begin{equation*}
    \lim_{u\to\infty}\frac{f(u)}{f'(u)u}=\lim_{u\to\infty}\frac{f(u)/f'(u)}{u}=\lim_{u\to\infty}\left(1-\frac{f(u) f''(u)}{f'^2(u)}\right)=\frac{q-1}{q}.
\end{equation*}
Here, we remark that when $q=1$, we have $uf'(u)/f(u)\to \infty$. Then, thanks to L'Hospital's rule, we obtain
\begin{equation}
\label{two}
    \lim_{u\to \infty}\frac{uf(u)}{\int_{0}^{u}f(t)\,dt}=\lim_{u\to \infty}\frac{uf'(u)}{f(u)}+1=
\begin{cases}
    \frac{2q-1}{q-1} &\text{if $q>1$,}\\
    \infty &\text{if $q=1$}.
\end{cases}    
\end{equation}
Since $q<\frac{N+2}{4}$, we can deduce that $\frac{2q-1}{q-1}>\frac{2N}{N-2}$. Hence, we have (i) with some $p_0>\frac{N+2}{N-2}$. 

Then, we prove (ii). By using (i) and Gronwall's inequality, we have
\begin{equation*}
 \int_{0}^{u}f(t) \,dt\ge\left(\frac{u}{u_1}\right)^{p_{0}+1} \int_{0}^{u_1}f(t)\,dt \hspace{4mm}\text{for all $u>u_1$}
\end{equation*}
with some $u_1>0$ sufficiently large. By using the three estimates above and taking $u_0>u_1$ sufficiently large, we obtain
\begin{equation}
\label{three}
   u^{p_{0}-1}\le C_0 f'(u) \hspace{4mm}\text{for all $u>u_0$}
\end{equation}
with some $C_0>0$. Thus, the conclusion follows.

Finally, we prove (iii). Thanks to \eqref{asf1}, \eqref{two} and \eqref{three}, we obtain 
\begin{equation*}
   u^{2-p_0} f''(u)=u^{1-p_0}\frac{f(u)f''(u)}{f'(u)^2}\cdot\frac{f'(u)u}{f(u)}\cdot f'(u)\ge \frac{1+o(1)}{qu^{p_0-1}}f'(u)\ge \frac{1+o(1)}{C_0q}
\end{equation*}
as $u\to\infty$. Thus, we obtain the result.
\end{proof}
Next, we find some examples of $f$
satisfying \eqref{asf0}, \eqref{asf1} and \eqref{asF1}. In order to obtain the examples, we prove the following 
\begin{lemma}
\label{yoilem}
Let $u_0>0$ and $p>1$. We assume that $b\in C^3(u_0,\infty)$ satisfies $b>0$, $b'\neq 0$ and $b''\neq 0$ for all $u\in (u_0,\infty)$. Then the following are satisfied.
\begin{itemize}
\item[(i)]Assume that $b$ satisfies $b'>0$ in $(u_0,\infty)$ and $b\to\infty$ as $u\to\infty$. Moreover, we suppose that 
\begin{equation}
\label{b1}
    \text{$b^{-1}\log b'\to 0$,\hspace{4mm}$c\to 0$ \hspace{4mm}and\hspace{4mm} $\frac{c'}{b'c}\to 0$ \hspace{4mm} as $u\to\infty$,}
\end{equation}
where $c=b''b'^{-2}$. Then,
$f(u)=e^{b(u)}$ satisfies 
\begin{equation*}
    Ff'-1=-(1+o(1))c\hspace{4mm}\text{and}\hspace{4mm}-\log F=(1+o(1))b \hspace{4mm}\text{as $u\to\infty$,}
\end{equation*}
where $F$ is that in \eqref{asf1}.
\item[(ii)]Assume that $b$ satisfies 
\begin{equation}
\label{b2}
    \text{$\frac{\log b}{\log u}\to 0$,\hspace{4mm}$cb\to 0$\hspace{4mm}and\hspace{4mm} $c's/c\to 0$\hspace{4mm} as $u\to\infty$,}
\end{equation}
where $c=b'ub^{-2}$. Then, $f=u^{p}b(u)$ satisfies
\begin{equation*}
    Ff'-\frac{p}{p-1}=-\frac{(1+o(1))}{(p-1)^{2}}bc \hspace{2mm}\text{and}\hspace{2mm}-\log F=(p-1+o(1))\log u 
\end{equation*}
as $u\to\infty$.
\end{itemize}
\end{lemma}

\begin{proof}
At first, we prove (i). From \eqref{b1}, we have $e^{-b}b'^{-1}=o(1)$ and $c=o(1)$ as $u\to\infty$. Hence, we obtain   
\begin{align*}
F(u)=\int_{u}^{\infty}\frac{\,ds}{f(s)}&=[-e^{-b}b'^{-1}]^{\infty}_{u}-\int_{u}^{\infty}e^{-b}c\,ds\\
&=e^{-b}b'^{-1}(1-c)-\int_{u}^{\infty}e^{-b}(cb'^{-1})'\,ds.
\end{align*}
Moreover, it follows from \eqref{b1} that $(cb'^{-1})'/c= -c+c'b'^{-1}c^{-1}=o(1)$ as $u\to\infty$. As a result, we have
\begin{equation*}
F(u)=e^{-b}b'^{-1}(1-(1+o(1))c) \hspace{4mm}\text{as $u\to\infty$.}
\end{equation*}
Thus, by using \eqref{b1} again, we obtain $-\log F(u)=(1+o(1)) b$ as $u\to\infty$. Moreover, since $f'=e^b b'$, we have
\begin{equation*}
f'(u)F(u)-1=-(1+o(1))c \hspace{4mm}\text{as $u\to\infty$.}
\end{equation*}
We now prove (ii). From \eqref{b2}, we have $u^{1-p}b^{-1}=o(1)$ and $bc=o(1)$ as $u\to\infty$. Therefore, 
we have
\begin{align*}
F(u)&=\int_{u}^{\infty}\frac{\,ds}{f(s)}=[\frac{-1}{p-1}s^{-(p-1)}b^{-1}]^{\infty}_{u}-\int_{u}^{\infty}\frac{1}{p-1} s^{-p}c\,ds\\
&=\frac{1}{p-1}u^{-(p-1)}b^{-1}(1-\frac{1}{p-1}bc)-\frac{1}{(p-1)^2}\int_{u}^{\infty}(sc') s^{-p}\,ds.
\end{align*}
Moreover, we deduce from \eqref{b2} that 
\begin{equation*}
    F(u)=\frac{1}{p-1}u^{-(p-1)}b^{-1}(1-\frac{1+o(1)}{p-1}bc) \hspace{4mm}\text{as $u\to\infty$.}
\end{equation*}
In particular, by using \eqref{b2} again, we obtain
\begin{equation*}
-\log F= (p-1+o(1))\log u \hspace{4mm}\text{as $u\to\infty$.}
\end{equation*}
In addition, since $f'(u)=pu^{p-1}b(1+bcp^{-1})$, we can deduce that
\begin{equation*}
    Ff'-\frac{p}{p-1}=\frac{p}{p-1}(1+p^{-1}bc)(1-\frac{1+o(1)}{p-1}bc)-\frac{p}{p-1}=\frac{-(1+o(1))bc}{(p-1)^2}
\end{equation*}
as $u\to\infty$.
\end{proof}
By using Lemma \ref{yoilem}, we obtain
\begin{lemma}
\label{examplelem}
In the case $N=10$, $f$ satisfies \eqref{asf0}, \eqref{asf1} and \eqref{asF1} with $q=q_{\mathrm{JL}}=1$ and
\begin{itemize}
    \item[(i)] $k=1-\frac{1}{\nu}$ and $\gamma=\frac{1}{\nu}$ if $f(u)=e^{(\log (1+u))^{\nu}}$ with $\nu>1$;
    \item[(ii)] $k=1$ and $\gamma=\frac{1-\nu}{\nu}$ if $f=e^{(1+u)^\nu}$ with $0<\nu<1$ or $1<\nu$;
    \item[(iii)] $k=2-\nu$ and $\gamma=\beta \nu(1-\nu)$
    if $f=e^{u+\beta (1+u)^\nu}$ with $0<\nu<1$;
    \item[(iv)] $k=1$ and $\gamma=-1$  if $f=e^{e^u}$;
    \item[(v)] $k=2$ and $\gamma=\nu$ if $f=e^u(1+u)^{\nu}$ with $\nu\in \mathbb{R}$.
\end{itemize}
In addition, in the case $N\ge 11$, $f$ satisfies \eqref{asf0}, \eqref{asf1} and \eqref{asF1} with $q=q_{\mathrm{JL}}$ and
\begin{itemize}
    \item[(vi)] $k=1+\nu$ and $\gamma=\nu \beta(p_{\mathrm{JL}}-1)^{\nu-1}$ if $f=(1+u)^{p_{\mathrm{JL}}}(1+\beta (\log (u+2))^{-\nu})$ with $0<\nu\le 1$ and $\beta>-(\log 2)^\nu$; 
    \item[(vii)] $k=1-\nu$ and $\gamma=-\nu\beta(p_{\mathrm{JL}}-1)^{-\nu-1}$
    if $f=(1+u)^{p_{\mathrm{JL}}}e^{\beta(\log (u+1))^\nu}$ with $0<\nu<1$ and $\beta\in \mathbb{R}$;
    \item[(viii)] $k=1$ and $\gamma=-\nu(p_{\mathrm{JL}}-1)^{-1}$
    if $f=(1+u)^{p_{\mathrm{JL}}}(\log (2+u))^\nu$ with $\nu\in \mathbb{R}$. 
\end{itemize}
\end{lemma}
\begin{proof}
We can easily confirm that all of the above nonlinearities $f$ satisfy \eqref{asf0} and \eqref{asf1} with some $q\ge 1$ depending on $f$. Moreover, by Remark \ref{Hito}, we have $q=q_{\mathrm{JL}}$ if \eqref{asf1} and \eqref{asF1} are satisfied. Hence, it is sufficient to prove \eqref{asF1}. In addition, since $u\to\infty$ if and only if $(u-1)\to\infty$, It suffices to prove 
\begin{equation*}
   \lim_{u\to\infty} (-\log G)^k (Gg'-q_{\mathrm{JL}})= \gamma, \hspace{2mm}\text{where $g(u):=f(u-1)$ and $G(u)=\int_{u}^{\infty}\frac{\,ds}{g(s)}$.}
\end{equation*} 

\textit{Proof of (i).} We define $b=(\log u)^{\nu}$. By a direct calculation, we have
\begin{equation*}
    b'=\nu(\log u)^{\nu-1}u^{-1},\hspace{2mm}b''=\nu u^{-2}(\log u)^{\nu-2}(\nu-1-\log u),\hspace{2mm}c=\frac{-(\log u-\nu+1)}{\nu(\log u)^{\nu}}.
\end{equation*}
Thus, we obtain $b'>0$, $b''\neq 0$ in $(u_0,\infty)$ with some $u_0>0$ and $b\to\infty$ as $u\to\infty$. Moreover, since
\begin{equation*}
  c^{-1} c'b'^{-1}= -(\nu-1)(1+o(1)) u^{-1}(\log u)^{-1} b'^{-1}=o(1)\hspace{4mm}\text{as $u\to\infty$},
\end{equation*}
we obtain \eqref{b1}. Hence, it follows from Lemma \ref{yoilem} that 
\begin{equation*}
    (-\log G)^{k}(Gg'-1)=-(1+o(1))b^k c=\frac{(1+o(1))}{\nu} \hspace{4mm}\text{as $u\to\infty$ with $k=\frac{\nu-1}{\nu}$}.
\end{equation*}

\textit{Proof of (ii).} We define $b=u^{\nu}$ with $0<\nu<1$ or $1<\nu$. By a direct calculation, we obtain
\begin{equation*}
    b'=\nu u^{\nu-1},\hspace{2mm}b''= \nu(\nu-1) u^{\nu-2},
    \hspace{2mm}c=\frac{\nu-1}{\nu}u^{-\nu}.
\end{equation*}
Hence, we have $b'>0$, $b''\neq 0$ in $(u_0,\infty)$ with some $u_0>0$ and $b\to\infty$ as $u\to\infty$. Moreover, since
\begin{equation*}
  c^{-1} c'b'^{-1}= -\nu u^{-1} b'^{-1}=o(1) \hspace{4mm}\text{as $u\to\infty$}, 
\end{equation*}
we can confirm \eqref{b1}. Thus,
it follows from Lemma \ref{yoilem} that 
\begin{equation*}
    (-\log G)(Gg'-1)=-(1+o(1))bc=\frac{(1-\nu)(1+o(1))}{\nu} \hspace{4mm}\text{as $u\to\infty$}.
\end{equation*}

\textit{Proof of (iii).} Without loss of generality, we assume that $\beta \neq 0$. 
We define $b=u-1+\beta u^{\nu}$ with $0<\nu<1$.  By a direct calculation, we get
\begin{equation*}
    b'=1+\beta \nu u^{\nu-1},\hspace{2mm}b''= \beta \nu(\nu-1) u^{\nu-2},
    \hspace{2mm}c=\beta \nu(\nu-1)u^{\nu-2}(1+\beta \nu u^{\nu-1})^{-2}.
\end{equation*}
Hence, we have $b'>0$, $b''\neq 0$ in $(u_0,\infty)$ with some $u_0>0$ and $b\to\infty$ as $u\to\infty$. Moreover, since
\begin{equation*}
  c^{-1} c'b'^{-1}= (\nu-2+o(1)) u^{-1} b'^{-1}=o(1)\hspace{4mm}\text{as $u\to\infty$},
\end{equation*}
we obtain \eqref{b1}. Thus,
it follows from Lemma \ref{yoilem} that 
\begin{equation*}
    (-\log G)^{k}(Gg'-1)=-(1+o(1))b^k c=\beta (1-\nu)\nu(1+o(1))
\end{equation*}
as $u\to\infty$ with $k=2-\nu$.

\textit{Proof of (iv).} We define $b=e^u$. By a direct calculation, we get
\begin{equation*}
    b'=b''=e^u,\hspace{2mm}c=e^{-u}.
\end{equation*}
Hence, we have $b'>0$, $b''\neq 0$ in $(u_0,\infty)$ with some $u_0>0$ and $b\to\infty$ as $u\to\infty$. Moreover, since
\begin{equation*}
  c^{-1} c'b'^{-1}= -b'^{-1}=o(1)\hspace{4mm}\text{as $u\to\infty$},
\end{equation*}
we obtain \eqref{b1}. Therefore, it follows from Lemma \ref{yoilem} that 
\begin{equation*}
    (-\log F)(Ff'-1)=-(1+o(1))bc=-1+o(1) \text{ as $u\to\infty$}.
\end{equation*}

\textit{Proof of (v).} Without loss of generality, we assume that $\nu\neq 0$.
We define $b=u-1+\nu(\log u)$. By a direct calculation, we get
\begin{equation*}
    b'=1+\nu u^{-1}, \hspace{2mm} b''=-\nu u^{-2},\hspace{2mm} c=-\nu u^{-2}(1+\nu u^{-1})^{-2}.
\end{equation*}
Thus, we can deduce that $b'>0$, $b''\neq 0$ in $(u_0,\infty)$ with some $u_0>0$ and $b\to\infty$ as $u\to\infty$. Moreover, since
\begin{equation*}
  c^{-1} c'b'^{-1}=-2(1+o(1))u^{-1}b'^{-1}=o(1)\hspace{4mm}\text{as $u\to\infty$},
\end{equation*}
we can confirm \eqref{b1}. Therefore, it follows from Lemma \ref{yoilem} that 
\begin{equation*}
    (-\log G)^{2}(Gg'-1)=-(1+o(1))b^2c= \nu(1+o(1))\text{ as $u\to\infty$}.
\end{equation*}

\textit{Proof of (vi).} Without loss of generality, we assume that $\beta\neq 0$.
We define $b=1+\beta(\log (1+u))^{-\nu}$. By a direct calculation, we obtain
\begin{equation*}
    b'=-\beta \nu (\log (1+u))^{-(\nu+1)}(1+u)^{-1}, \hspace{2mm}c=-\nu\beta(\log (1+u))^{-(\nu+1)}\frac{u}{(1+u)b^2}.
\end{equation*}
Hence, we can confirm that $b'\neq 0$, $b''\neq 0$ in $(u_0,\infty)$ with some $u_0>0$. Moreover, since 
\begin{equation*}
    c's/c=-(\nu+1+o(1))((\log (1+u))^{-1}-2b^{-1}b'u)=o(1)\hspace{4mm}\text{as $u\to\infty$}, 
\end{equation*}
we obtain \eqref{b2}. Thus, it follows from Lemma \ref{yoilem} that 
\begin{align*}
(-\log G)^{k}(Gg'-q_{\mathrm{JL}})&=-(1+o(1))(p_{\mathrm{JL}}-1)^{k-2}(\log u)^k bc\\
&= \nu(1+o(1))(p_{\mathrm{JL}}-1)^{k-2}\beta
\end{align*}
as $u\to\infty$ with $k=\nu+1$.

\textit{Proof of (vii).} Without loss of generality, we assume that $\beta\neq 0$. We define $b=1+e^{\beta(\log u)^\nu}$.
By a direct calculation, we can deduce that
\begin{equation*}
    b'=\beta \nu (\log u)^{\nu-1}u^{-1} (b-1),\hspace{2mm}c=\nu\beta(\log u)^{\nu-1} \frac{b-1}{b^2}.
\end{equation*}
Therefore, we can confirm that $b'\neq 0$, $b''\neq 0$ in $(u_0,\infty)$ with some $u_0>0$. Moreover, since 
\begin{equation*}
    c's/c=(\nu-1)(\log u)^{-1}-(1+o(1))b^{-1}b'u=o(1)\hspace{4mm}\text{as $u\to\infty$}, 
\end{equation*}
we have \eqref{b2}. Hence, it follows from Lemma \ref{yoilem} that 
\begin{align*}
(-\log G)^{k}(Gg'-q_{\mathrm{JL}})&=-(1+o(1))(p_{\mathrm{JL}}-1)^{k-2}(\log u)^k bc\\
&= -\nu\beta (1+o(1))(p_{\mathrm{JL}}-1)^{k-2}
\end{align*}
as $u\to\infty$ with $k=1-\nu$.

\textit{Proof of (viii).} Without loss of generality, we assume that $\nu\neq 0$.
We define $b=(\log (1+u))^{\nu}$.
By a direct calculation, we get
\begin{equation*}
    b'=\nu (\log (1+u))^{\nu-1}(1+u)^{-1},\hspace{2mm}c=\nu(\log (1+u))^{-\nu-1}\frac{u}{1+u}.
\end{equation*}
Thus, we can confirm that $b'\neq 0$, $b''\neq 0$ in $(u_0,\infty)$ with some $u_0>0$. Moreover, since 
\begin{equation*}
    c's/c=-(\nu+1+o(1))(\log (1+u))^{-1}=o(1)\hspace{4mm}\text{as $u\to\infty$}, 
\end{equation*}
we obtain \eqref{b2}. Thus,
it follows from Lemma \ref{yoilem} that 
\begin{equation*}
(-\log G)(Gg'-q_{\mathrm{JL}})=-(1+o(1))(p_{\mathrm{JL}}-1)^{-1}(\log u) bc= -\nu(1+o(1))(p_{\mathrm{JL}}-1)^{-1}
\end{equation*}
as $u\to\infty$.
\end{proof}
\section{Asymptotic behavior of the singular solution}
In this section, we study the asymptotic behavior and the Morse index of the radial singular solution by using the specific transformation used in \cite{Mi2023} and the fixed point theorem. As a result, we prove Theorem \ref{singularthm}. 

In the following, we always assume that $f$ is extended to $\mathbb{R}$ such that $f\in C^2(\mathbb{R})$ and $f>0$ in $\mathbb{R}$. Then, we can extend a radial  solution  $v\in C^2(0,\varepsilon)$ of
$-\Delta v=f(v)$ on $(0,\varepsilon)$
to $(0,\infty)$ so that $v\in C^2(0,\infty)$ and $v$ satisfies the equation on $(0,\infty)$.

We first quote the following result introduced in \cite{Mi2023}.
\begin{theorem}[see \cite{Mi2023}]
\label{senthm}
Let $N\ge 3$. We assume that $f$ satisfies \eqref{asf0} and \eqref{asf1}. Then, there exists the unique radial singular solution $V\in C^2(0,\infty)$ of 
\begin{equation}
\label{singulareq}
  -\Delta V=f(V) \hspace{4mm}\text{in $(0,\infty)$}
\end{equation}
such that 
\begin{equation*}
    V(r)=F^{-1}(\frac{r^2}{2N-4q}(1+o(1))) \hspace{4mm}\text{as $r\to 0$}.
\end{equation*}
Moreover, the solution $v=v(r,\alpha)\in C^2(0,\infty)\cap C^{0}[0,\infty)$ of the following equation
\begin{equation}
\label{gv}
\left\{
\begin{alignedat}{4}
&v''+\frac{N-1}{r}v'+f(v)=0, \hspace{14mm}0<r<\infty,\\
&v(0)=\alpha, \hspace{2mm} v'(0)=0
\end{alignedat}
\right.
\end{equation}
satisfies 
\begin{equation*}
    v(r,\alpha)\to V \hspace{4mm}\text{in $C^2_{\mathrm{loc}}(0,\infty)$ as $\alpha\to\infty$.} 
\end{equation*}
\end{theorem}
From Theorem \ref{senthm}, we obtain   
\begin{lemma}
\label{philem}
We assume that $f$ satisfies
\eqref{asf0}, \eqref{asf1} and \eqref{asF1} with $0<k\le 2$ and $\gamma\in \mathbb{R}$. Let $V$ be the radial singular solution of \eqref{singulareq}. Then, it follows that
\begin{equation*}
  \phi(V):=f'(V)F(V)-q_{\mathrm{JL}}=\frac{\gamma+o(1)}{2^k(-\log r)^k}\hspace{4mm}\text{as $r\to 0$.}
\end{equation*}
\end{lemma}
Now, we assume that $f$ satisfies \eqref{asf0}, \eqref{asf1} and \eqref{asF1}. As mentioned in Remark \ref{Hito}, we obtain $q=q_{\mathrm{JL}}$. Hence, Theorem \ref{senthm} and Lemma \ref{philem} imply that 
\begin{equation*}
    f'(V)= \frac{f'(V)F(V)}{F(V)}=\frac{q_{\mathrm{JL}}(2N-4q_\mathrm{JL})}{r^2}(1+o(1))=\frac{(N-2)^2}{4r^2}(1+o(1)) \hspace{4mm}\text{as $r\to 0$}.
\end{equation*}
However, as mentioned in the introduction, this estimate is not sufficient to study the Morse index of the radial singular solution, and thus a refined asymptotic profile of $V$ is needed. Hence, in the following, we aim to prove 
\begin{theorem}
\label{asythm}
We assume that $f$ satisfies \eqref{asf0}, \eqref{asf1} and \eqref{asF1} with $0<k\le 2$ and $\gamma\in \mathbb{R}$. Let $V$ be the unique radial singular solution of \eqref{singulareq}.
Then, we have 
\begin{equation*}
    f'(V)=\frac{(N-2)^2}{4r^2}+\frac{\gamma\sqrt{N-1}+o(1)}{2^{k-2}r^2(-\log r)^k} \hspace{4mm}\text{as $r\to 0$.}
\end{equation*}
\end{theorem}
\subsection{The case \texorpdfstring{$N=10$}{LG}}
In this subsection, we assume that $N=10$ and
$f$ satisfies \eqref{asf0}, \eqref{asf1} and \eqref{asF1}.
We define $x\in C^2(0,\infty)$ as 
\begin{equation*}
\label{defx}
\frac{F(V)}{r^2}=\frac{e^{-x(t)}}{2N-4}\hspace{4mm}\text{with $t=-\log r$}.
\end{equation*}
Then, thanks to \cite[Proposition 3.1]{Mi2023},
we can deduce that $x$ satisfies
\begin{equation*}
\label{eqx}
x''-(N-2)x'+(2N-4)(e^x-1)+\phi(V)(x'+2)^2=0,
\end{equation*}
\begin{equation}
\label{cx}
x(t)\to 0 \hspace{4mm}\text{and}\hspace{4mm} x'(t)\to 0\hspace{4mm} \text{as $t\to\infty$.}
\end{equation}
Then, we first prove the following lemma, which implies that $x(t)=O(t^{-k})$.
\begin{lemma}
\label{xlem}
Let $N=10$. We assume that \eqref{asf0}, \eqref{asf1} and \eqref{asF1} with $0<k\le 2$ and
$\gamma\in \mathbb{R}$. Then, $t^k |x(t)|\le 2^{-k}(1+|\gamma|)$ for all $t\in (t_0,\infty)$ with some $t_0>0$.
\end{lemma}

\begin{proof}
We define $\phi_1(t):=\phi(V)(x'+2)^2$. Then, $x\in C^2(0,\infty)$ is a bounded solution of 
\begin{equation}
\label{as1}
    x''-8x'+16(e^{x}-1)+\phi_1(t)=0.
\end{equation}
We claim that $x\in C^2(0,\infty)$ is the unique bounded solution of \eqref{as1} satisfying $x(t)\to 0$. Indeed, if $y$ is another bounded solution of \eqref{as1} satisfying $y(t)\to 0$ as $t\to\infty$, then $z:=y-x$ satisfies 
\begin{equation*}
    z''-8z'+16 e^{x}\frac{e^{z}-1}{z} z=0.
\end{equation*}
Thus, thanks to \cite[Lemma 4.2]{deng2000}, we have $z=0$.

Now, we construct a bounded solution $y\in C^2(t_0, \infty)$ of \eqref{as1} with $y(t)\to 0$ as $t\to 0$. We first find $t_0>0$ such that
\begin{equation}
\label{kei1}
    \frac{e(1+|\gamma|)}{(2t)^k}<\frac{1}{2}, \hspace{4mm} \frac{(1+|\gamma|)2^{-k}e}{t^k}\le \frac{1}{2} \hspace{4mm} \text{for all $t>t_0$}
\end{equation}
and
\begin{equation}
\label{hphi1}
|\phi_1(t)|=\frac{|\gamma|+o(1)}{2^{k}t^{k}}(x'(t)+2)^2\le \frac{|\gamma|+1}{2^{k-3}t^{k}} \hspace{4mm} \text{for all $t>t_0$}
\end{equation}
by using Lemma \ref{philem} and \eqref{cx}, where $o(1)\to 0$ as $t\to\infty$.
Then, we define 
\begin{equation*}
\mathcal{X}_1=:\{y\in C[t_0,\infty); \lVert y \rVert:=\sup_{t>t_0} t^{k}|y(t)| \le (|\gamma|+1)2^{-k}\} 
\end{equation*}
and
\begin{equation*}
    \Phi(y(t)):=\int_{t}^{\infty} (t-s)e^{4(t-s)}[16(e^{y(s)}-1-y(s))+\phi_1(s)]\,ds.
\end{equation*}
Let $y,z\in \mathcal{X}_1$. Then, it follows from \eqref{kei1} that 
\begin{align}
\label{hy1}
\begin{split}
    |(e^{y(t)}-1-y(t))&-(e^{z(t)}-1-z(t))|=|e^{y(t)}-e^{z(t)}-(y(t)-z(t))|\\
&=\frac{1}{2}e^{y+\theta(z-y)}(y-z)^2 \hspace{4mm}\text{with some $0\le \theta(t)\le 1$}\\
&\le \frac{(1+|\gamma|)e|y-z|}{(2t)^k}.
\end{split}
\end{align}
Thus, by combining \eqref{kei1}, \eqref{hphi1} and \eqref{hy1}, we can deduce that
\begin{align*}
   t^k |\Phi(y(t))|&\le t^k\int_{t}^{\infty}(s-t)e^{-4(s-t)} \left(\frac{e(1+|\gamma|)}{2^{k-4}s^k}|y(s)|+\frac{1+|\gamma|}{2^{k-3}s^k}\right)\,ds\\
  &\le (1+|\gamma|)\left(\frac{e(1+|\gamma|)2^{4-2k}}{t^k}+ 2^{3-k}\right)\int_{t}^{\infty}(s-t)e^{-4(s-t)}\,ds\\
  &= (1+|\gamma|)\left(\frac{e(1+|\gamma|)2^{4-2k}}{t^k}+ 2^{3-k}\right)\left[\frac{(s-t)}{4}e^{-4(s-t)}+\frac{e^{-4(s-t)}}{16}\right]^{s=t}_{\infty}\\
  &=(1+|\gamma|)\left(\frac{(1+|\gamma|)e 2^{-2k}}{t^k}+ 2^{-1-k}\right)\le 2^{-k}(1+|\gamma|)
\end{align*}
and
\begin{align*}
t^k|\Phi(y(t))-\Phi(z(t))|&\le t^k\int_{t}^{\infty}\frac{e(1+|\gamma|)}{2^{k-4}s^k}(s-t)e^{-4(s-t)} |y(s)-z(s)|\,ds\\
&\le \int_{t}^{\infty}\frac{e(1+|\gamma|)}{2^{k-4}s^k}(s-t)e^{-4(s-t)} \lVert y-z\rVert \,ds\\
&= \frac{e(1+|\gamma|)}{2^k t^k} \lVert y-z\rVert \le \frac{1}{2}\lVert y-z\rVert
\end{align*}
for all $t>t_0$. 
Therefore, by the fixed point theorem, we can obtain a solution $y(t)\in C^2(t_0,\infty)$ of \eqref{as1} such that 
$t^k |y|\le 2^{-k}(1+|\gamma|)$ for any $t>t_0$. Since the bounded solution of \eqref{as1} is unique, we obtain the result.
\end{proof}
By using Lemma \ref{xlem}, we obtain the following
\begin{proposition}
\label{superprop}
Under the same assumption of Lemma \ref{xlem},
we have 
\begin{equation*}
    t^k\left(x(t)+\frac{\gamma}{2^{k+2}t^k}\right)\to 0 \hspace{4mm}\text{as $t\to\infty$.}
\end{equation*}
\end{proposition}
\begin{proof}
We define 
\begin{equation*}
    \phi_2(t):=\phi_1(t)-\frac{\gamma}{2^{k-2}t^k}\hspace{4mm}\text{and} \hspace{4mm}y=x+\frac{\gamma}{2^{k+2}t^{k}},
\end{equation*}
where $\phi_1(t)$ is that in the proof of Lemma \ref{xlem}. Then, $y$ is a solution of 
\begin{equation}
\label{as2}
   y''-8y'+16y+16(e^x-1-x)+\phi_2(t)+\phi_3(t)=0,
\end{equation}
where 
\begin{equation*}
    \phi_3(t)= 8\left(\frac{\gamma}{2^{k+2}t^{k}}\right)'-\left(\frac{\gamma}{2^{k+2}t^{k}}\right)''.
\end{equation*}
Here, by a similar argument to that in the proof of Lemma \ref{xlem}, we can deduce that $y$ is the unique bounded solution of \eqref{as2} satisfying $y(t)\to 0$. Thus, we can verify that
\begin{equation*}
    y(t)=\int_{t}^{\infty} (t-s)e^{4(t-s)}[16(e^x-1-x)+\phi_2(s)+\phi_3 (s)]\,ds.
\end{equation*}
Let $\varepsilon>0$. Then, thanks to Lemma \ref{xlem}, there exists $0\le \theta(t)\le 1$ such that 
\begin{align*}
    |e^{x}-1-x|\le e^{\theta x}x^2\le \frac{e(1+|\gamma|)^2}{(2t)^{2k}}\le \frac{\varepsilon}{2t^k} 
\end{align*}
for all $t$ sufficiently large. Moreover, thanks to Lemma \ref{philem} and \eqref{cx}, we obtain 
\begin{equation*}
|\phi_2(t)|+|\phi_3(t)|\le \frac{\varepsilon}{t^k}  
\end{equation*}
for all $t$ sufficiently large. We choose $t_0>0$ such that the above estimates are satisfied for all $t>t_0$. Then, it follows from the above two estimates that
\begin{align*}
   t^k|y(t)|&\le t^k\int_{t}^{\infty}(s-t)e^{-4(s-t)} \left(\frac{8\varepsilon}{s^{k}}+\frac{\varepsilon}{s^k}\right)\,ds\\
  &\le 9\varepsilon
  \int_{t}^{\infty}(s-t)e^{-4(s-t)}\,ds\\
  &=\frac{9\varepsilon}{16}\le \varepsilon
\end{align*}
for all $t>t_0$ if $t_0$ is sufficiently large. Therefore, the conclusion follows.
\end{proof}
\begin{proof}[Proof of Theorem \ref{asythm} in the case $N=10$.]
Thanks to Lemma \ref{philem} and Proposition \ref{superprop}, it follows that 
\begin{align*}
    f'(V)&=\frac{1}{F(V)}\left(1+\frac{\gamma+o(1)}{(-2\log r)^k}\right)\\
    &=16e^{x(t)}r^{-2}\left(1+\frac{\gamma+o(1)}{(-2\log r)^k}\right)\\
    &=16r^{-2}\left(1+\frac{\gamma+o(1)}{(-2\log r)^k}\right)\exp\left[-\frac{\gamma+o(1)}{2^{k+2}(-\log r)^k}\right]\\
    &=16r^{-2}\left(1+\frac{\gamma+o(1)}{(-2\log r)^k}\right)
    \left(1-\frac{\gamma+o(1)}{2^{k+2}(-\log r)^k}\right)\\
    &=16r^{-2}+\frac{3\gamma+o(1)}{2^{k-2}r^2(-\log r)^k} \hspace{4mm}\text{as $r\to 0$.}
\end{align*}
Hence, the conclusion follows from the fact that $q_{\mathrm{JL}}=1$ if $N=10$. 
\end{proof}
\subsection{The case \texorpdfstring{$N\ge 11$}{LG}} In this subsection, we assume that $N\ge 11$ and
$f$ satisfies \eqref{asf0}, \eqref{asf1} and \eqref{asF1}. We define $z\in C^2(0,\infty)$ by 
\begin{equation*}
\label{defz}
\frac{F(V)}{r^2}=\frac{(1+z(t))^{-1/(q_{\mathrm{JL}}-1)}}{2N-4q_{\mathrm{JL}}}\hspace{4mm}\text{with $t=-\log r$}.
\end{equation*}
Then, thanks to \cite[Proposition 3.1]{Mi2023},
we can deduce that $z$ satisfies
\begin{align*}
\label{eqz}
z''-(N&+2-4q_{\mathrm{JL}})z'+(q-1)(2N-4q_{\mathrm{JL}})((z+1)^{p_{\mathrm{JL}}}-1-z)\\
&+(q_{\mathrm{JL}}-1)\phi(V)\left(\frac{z'}{(q_{\mathrm{JL}}-1)(z+1)}+2\right)^2(z+1)=0,
\end{align*}
\begin{equation}
\label{cz}
z(t)\to 0 \hspace{4mm}\text{and}\hspace{4mm} z'(t)\to 0\hspace{4mm} \text{as $t\to\infty$,}
\end{equation}
where $p_{\mathrm{JL}}=q_{\mathrm{JL}}/(q_{\mathrm{JL}}-1)$. We prove the following
\begin{lemma}
\label{zlem}
Let $N>10$. We assume that $f$ satisfies \eqref{asf0}, \eqref{asf1} and \eqref{asF1} with $q=q_{\mathrm{JL}}$ and
some $0<k\le 2$ and
$\gamma\in \mathbb{R}$. Then, $t^k z(t)$ is bounded in $(t_0,\infty)$ with some $t_0>0$.
\end{lemma}

\begin{proof}
We define 
\begin{equation*}
    \phi_1(t):=(q_{\mathrm{JL}}-1)\phi(V)\left(\frac{z'}{(q_{\mathrm{JL}}-1)(z+1)}+2\right)^2(z+1).
\end{equation*}
Then, $z\in C^2(0,\infty)$ is a bounded solution of 
\begin{equation}
\label{as1-2}
    z''- 2az'+\frac{a^2}{p_{\mathrm{JL}}-1}((z+1)^{p_{\mathrm{JL}}}-1-z)+\phi_1(t)=0
\end{equation}
with $a=\frac{N+2-4q_{\mathrm{JL}}}{2}=1+\sqrt{N-1}>0$.

We claim that $z\in C^2(0,\infty)$ is the unique bounded solution of \eqref{as1-2} satisfying $z(t)\to 0$ as $t\to\infty$. Indeed, if $y$ is another bounded solution of \eqref{as1-2} satisfying $y(t)\to 0$ as $t\to\infty$, then $w:=y-z$ satisfies 
\begin{equation*}
    w''-2aw'+\frac{a^2}{p_{\mathrm{JL}}-1} \left(\frac{(y+1)^{p_{\mathrm{JL}}}-(z+1)^{p_{\mathrm{JL}}}}{y-z}-1\right) w=0.
\end{equation*}
Since
\begin{equation*}
    \frac{(y+1)^{p_{\mathrm{JL}}}-(z+1)^{p_{\mathrm{JL}}}}{y-z}=p_{\mathrm{JL}}(z+\theta w+1)^{p_{\mathrm{JL}}-1} \hspace{4mm}\text{with some $0\le \theta(t)\le 1$},
\end{equation*}
thanks to \cite[Lemma 4.2]{deng2000}, we have $w=0$.

Now, we construct a bounded solution $y\in C^2(t_0, \infty)$ of \eqref{as1-2} with $y(t)\to 0$ as $t\to 0$. Then, thanks to Lemma \ref{philem} and \eqref{cz}, we can take $t_0>0$ so that
\begin{equation}
\label{kei2}
  \frac{{C^*}^2 a^2 p_{\mathrm{JL}} 2^{|p_{\mathrm{JL}}-2|}}{t^k}<(1+|\gamma|)q_{\mathrm{JL}}2^{2-k},\hspace{4mm}\frac{C^{*}p_{\mathrm{JL}} 2^{|p_{\mathrm{JL}}-2|}}{t^k}<\frac{1}{2} \hspace{4mm}\text{for all $t>t_0$}
\end{equation}
and
\begin{equation}
\label{hphi1-2}
|\phi_1(t)|\le \frac{(|\gamma|+o(1))(q_{\mathrm{JL}}-1+o(1))}{2^{k-2}t^k}\le \frac{(1+|\gamma|)q_{\mathrm{JL}}}{2^{k-2}t^k}   \hspace{4mm}\text{for all $t>t_0$},
\end{equation}
where $C^{*}:=2^{3-k}a^{-2}(1+|\gamma|)q_{\mathrm{JL}}$ and $o(1)\to 0$ as $t\to \infty$.
We define 
\begin{equation*}
\mathcal{X}_1=\{y\in C[t_0,\infty); \lVert y \rVert:=\sup_{t>t_0} t^{k}|y(t)| \le  C^{*}\} 
\end{equation*}
and
\begin{equation*}
    \Phi(y(t)):=\int_{t}^{\infty} (t-s)e^{a(t-s)}\left[\frac{a^2}{p_{\mathrm{JL}}-1}\left((y(s)+1)^{p_{\mathrm{JL}}}-1-p_{\mathrm{JL}}y(s)\right)+\phi_1(s)\right]\,ds.
\end{equation*}
Let $y,w\in \mathcal{X}_1$. Then, 
it follows from \eqref{kei2} that there exists $0\le \theta(t)\le 1$ such that
\begin{align}
\label{hy1-2}
\begin{split}
   |((y(t)+1)^{p_{\mathrm{JL}}}-&1-p_{\mathrm{JL}}y(t))-((w(s)+1)^{p_{\mathrm{JL}}}-1-p_{\mathrm{JL}}w(t))|\\
   &=|(y+1)^{p_{\mathrm{JL}}}-(w+1)^{p_{\mathrm{JL}}}-p_{\mathrm{JL}}(y-w)|\\
   &=\frac{p_{\mathrm{JL}}(p_{\mathrm{JL}}-1)(y+\theta(w-y)+1)^{p_{\mathrm{JL}}-2}}{2}(y-w)^2 \\
   &\le p_{\mathrm{JL}}(p_{\mathrm{JL}}-1)2^{|p_{\mathrm{JL}}-2|-1}(y-w)^2\\
   &\le C^{*} p_{\mathrm{JL}}(p_{\mathrm{JL}}-1)2^{|p_{\mathrm{JL}}-2|} t^{-k}|y-w|
\end{split}
\end{align}
for all $t>t_0$ if $t_0$ is sufficiently large.
Thus, from \eqref{hy1-2} and \eqref{hphi1-2}, we obtain 
\begin{align*}
   t^k &|\Phi(y(t))|\le t^k\int_{t}^{\infty}(s-t)e^{-a(s-t)} \left(\frac{C^{*}p_{\mathrm{JL}}a^2 2^{|p_{\mathrm{JL}}-2|}}{s^k}|y(s)|+\frac{(1+|\gamma|) q_{\mathrm{JL}}}{2^{k-2}s^k}\right)\,ds\\
  &\le \left(\frac{{C^*}^2 p_{\mathrm{JL}}a^2 2^{|p_{\mathrm{JL}}-2|}}{t^k}+ (1+|\gamma|) q_{\mathrm{JL}}2^{2-k}\right)\int_{t}^{\infty}(s-t)e^{-a(s-t)}\,ds\\
  &= \left(\frac{{C^*}^2 p_{\mathrm{JL}}a^2 2^{|p_{\mathrm{JL}}-2|}}{t^k}+ (1+|\gamma|) q_{\mathrm{JL}}2^{2-k}\right)\left[\frac{(s-t)}{a}e^{-a(s-t)}+\frac{e^{-a(s-t)}}{a^2}\right]^{s=t}_{\infty}\\
  &=\frac{1}{a^2}\left(\frac{{C^*}^2 p_{\mathrm{JL}}a^2 2^{|p_{\mathrm{JL}}-2|}}{t^k}+ (1+|\gamma|) q_{\mathrm{JL}}2^{2-k}\right)\le \frac{2^{3-k}(1+|\gamma|)q_{\mathrm{JL}}}{a^2}
\end{align*}
and
\begin{align*}
t^k|\Phi(y(t))-\Phi(z(t))|&\le t^k\int_{t}^{\infty}\frac{C^{*}p_{\mathrm{JL}}a^2 2^{|p_{\mathrm{JL}}-2|}}{s^k}(s-t)e^{-a(s-t)} |y(s)-z(s)|\,ds\\
&\le \int_{t}^{\infty}\frac{C^{*}p_{\mathrm{JL}}a^2 2^{|p_{\mathrm{JL}}-2|}}{s^k}(s-t)e^{-a(s-t)} \lVert y-z\rVert \,ds\\
&= \frac{C^{*}p_{\mathrm{JL}} 2^{|p_{\mathrm{JL}}-2|}}{t^k} \lVert y-z\rVert \le \frac{1}{2}\lVert y-z\rVert
\end{align*}
for all $t>t_0$. 
Therefore, by the fixed point theorem, we can get a solution $y(t)\in C^2(t_0,\infty)$ of \eqref{as1-2} such that $t^k |y|\le 2^{3-k}a^{-2}(1+|\gamma|)q_{\mathrm{JL}}$. Since the bounded solution of \eqref{as1-2} is unique, we obtain the result.
\end{proof}
Thanks to Lemma \ref{zlem}, we obtain
\begin{proposition}
\label{superprop2}
Under the same assumptions of Lemma \ref{zlem}, we have 
\begin{equation*}
t^k\left(z(t)+\frac{\gamma(q_{\mathrm{JL}}-1)}{a^2 2^{k-2}t^k}\right)\to 0 \hspace{4mm} \text{as $t\to\infty$.}
\end{equation*}
\end{proposition}
\begin{proof}
We define 
\begin{equation*}
   \phi_2(t):=\phi_1(t)-\frac{(q_{\mathrm{JL}}-1)\gamma}{2^{k-2}t^k}\hspace{4mm} \text{and} \hspace{4mm} y=z+\frac{\gamma(q_{\mathrm{JL}}-1)}{a^2 2^{k-2}t^{k}},
\end{equation*}
where $\phi_1$ is that in the proof of Lemma \ref{zlem}. Then, $y$ is a solution of 
\begin{equation}
\label{as2-2}
   y''-2ay'+a^2 y +\frac{a^2}{p_{\mathrm{JL}}-1}((z+1)^{p_{\mathrm{JL}}}-1-p_{\mathrm{JL}}z)+\phi_2(t)+\phi_3(t)=0,
\end{equation}
where $a$ is that in the proof of Lemma \ref{zlem} and \begin{equation*}
\phi_3(t)=2a\left(\frac{\gamma(q_{\mathrm{JL}}-1)}{a^2 2^{k-2}t^{k}}\right)'-\left(\frac{\gamma(q_{\mathrm{JL}}-1)}{a^2 2^{k-2}t^{k}}\right)''.
\end{equation*}
Here, by a similar argument to that in the proof of Lemma \ref{zlem}, we can deduce that $y$ is the unique bounded solution of \eqref{as2-2} satisfying $y(t)\to 0$ as $t\to\infty$. 
Thus, we can confirm that
\begin{equation*}
    y(t):=\int_{t}^{\infty} (t-s)e^{a(t-s)}[\frac{a^2}{p_{\mathrm{JL}}-1}((z+1)^{p_{\mathrm{JL}}}-1-p_{\mathrm{JL}}z)+\phi_2(s)+\phi_3(s)]\,ds.
\end{equation*}
Let $\varepsilon>0$. Then, thanks to Lemma \ref{zlem}, we have
\begin{align*}
\frac{a^2}{p_{\mathrm{JL}}-1}|(z+1)^{p_{\mathrm{JL}}}-1-p_{\mathrm{JL}}z|\le a^2 p_{\mathrm{JL}} 2^{|p_{\mathrm{JL}}-2|}z^2\le \frac{\varepsilon}{t^k}
\end{align*}
for all $t$ sufficiently large. Moreover, thanks to Lemma \ref{philem} and \eqref{cz}, we obtain 
\begin{equation*}
|\phi_2(t)|+|\phi_3(t)|\le \frac{\varepsilon}{t^k}
\end{equation*}
for all $t$ sufficiently large. We choose $t_0$ such that the above two estimates are satisfied for $t>t_0$.
Therefore, it follows from the above two estimates that
\begin{align*}
   t^k |y(t)|&\le t^k \int_{t}^{\infty}(s-t)e^{-a(s-t)} \frac{2\varepsilon}{s^k}\,ds\\
  &\le 2\varepsilon
  \int_{t}^{\infty}(s-t)e^{-a(s-t)}\,ds \le \frac{2\varepsilon}{a^2}
\end{align*}
for all $t>t_0$ if $t_0$ is sufficiently large. Therefore, we obtain the result.
\end{proof}

\begin{proof}[Proof of Theorem \ref{asythm} in the case $N\ge 11$.]
Thanks to Lemma \ref{philem} and Proposition \ref{superprop2}, it follows that 
\begin{align*}
f'(V)&=\frac{1}{F(V)}\left(q_{\mathrm{JL}}+\frac{\gamma+o(1)}{(-2\log r)^k}\right)\\
&=a^2r^{-2}\left(q_{\mathrm{JL}}+\frac{\gamma+o(1)}{(-2\log r)^k}\right)(1+z(t))^{1/(q_{\mathrm{JL}}-1)}\\
    &=a^2 r^{-2}\left(q_{\mathrm{JL}}+\frac{\gamma+o(1)}{(-2\log r)^k}\right)\left(1-\frac{(q_{\mathrm{JL}}-1)(\gamma+o(1))}{2^{k-2}a^2(-\log r)^k}\right)^{1/(q_{\mathrm{JL}}-1)}\\
    &=\frac{(N-2)^2}{4r^2}+ 
    \frac{\gamma+o(1)}{2^k r^2(-\log r)^k}(a^2-4q_{\mathrm{JL}})\\
    &=\frac{(N-2)^2}{4r^2}+\frac{\gamma\sqrt{N-1}+o(1)}{2^{k-2}r^2(-\log r)^k} \hspace{6mm}\text{as $r\to 0$}.
\end{align*}
\end{proof}
\subsection{Proof of Theorem \ref{singularthm}}
We start from introducing the following lemma.
\begin{lemma}[see \cite{Mi2023} for example]
\label{basiclem}
Let $N\ge 3$. We assume that $f\in C^2(\mathbb{R})$ is non-negative. Let $v$ be a solution of \eqref{singulareq} satisfying $\liminf_{r\to\infty}v(r)>-\infty$. Then, we have $v'\le 0$ in $(0,\infty)$.
\end{lemma}
\begin{proof}
For readers' convenience, we state the proof. We prove $v'\le 0$ in $(0,\infty)$ by contradiction. Thus, we assume that $v'>0$ for some $t\in (0,\infty)$. Since $(r^{N-1}v')'=-r^{N-1}f(v)\le 0$ in $(0,\infty)$, we obtain $v'(s)\ge s^{1-N}t^{N-1}v'(t)$ for $s<t$. By integrating the inequality on $(r,t)$, we have 
\begin{equation*}
    -v(r)\ge-v(t)+\frac{v'(t)}{N-2}t^{N-1}(r^{2-N}-t^{2-N}) \to\infty
\end{equation*}
as $r\to 0$, which contradicts to the fact that $\liminf_{r\to\infty}v(r)>-\infty$.
\end{proof}
 
\begin{proposition}
\label{h1prop}
Assume that $f$ satisfies \eqref{asf0} and \eqref{asf1}. Let $V$ be the radial singular solution of \eqref{singulareq}. We define $\sqrt{\lambda_{*}}$ by the first zero of $V$. 
Then, we have $V\in H^1_{0}(B_{\sqrt{\lambda_{*}}})$.
\end{proposition}
\begin{proof}
Thanks to \eqref{sim1} and Theorem \ref{senthm}, we obtain 
\begin{equation*}
    f'(V)= \frac{f'(V)F(V)}{F(V)}=\frac{q(2N-4q)}{r^2}(1+o(1))\hspace{4mm}\text{as $r\to 0$}.
\end{equation*}
Hence, it follows from Lemma \ref{basiclem} and
Lemma \ref{convexlem} that there exist $r_0>0$ and $C>0$ independent of $r$ such that $V(r)\ge V(r_0)>u_0+1$
for $r<r_0$ and
\begin{equation*}
   V(r)^{p_0-1}\le C_{0}f'(V(r))\le \frac{C}{r^2} \hspace{4mm}\text{for any $r<r_0$,}
\end{equation*}
where $p_0>\frac{N+2}{N-2}$,
$C_0>0$ and $u_0>0$ are those in Lemma \ref{convexlem}. Moreover, thanks to Lemma \ref{convexlem},  Theorem \ref{senthm} and \cite[Lemma 4.2]{Mi2023}, we can deduce that 
\begin{equation*}
    0<-rV'(r)<\frac{2N}{(p_0+1)}V(r)
\end{equation*}
for all $0<r<r_0$. Combining the above two estimates, we have
\begin{equation*}
0\le -V'(r)\le \hat{C} r^{-(1+\frac{2}{p_0-1})} 
\end{equation*}
for all $0<r<r_0$ with some $\hat{C}>0$ independent of $r$. Hence, by using Theorem \ref{senthm} again, we have $V\in H^1_{0}(B_{\sqrt{\lambda_{*}}})$.
\end{proof}

Next, we introduce the following proposition, which plays a key role in studying the Morse index of the radial singular solution.
\begin{proposition}
\label{criticalprop}
Let $N\ge 2$. For any $\varphi\in C^{0,1}_{\mathrm{c}}(B_1)$, we can deduce that
\begin{equation}
\label{criticalhardy}
 \mathcal{I}(\varphi):=\int_{B_1}|\nabla \varphi|^2\,dx-\frac{(N-2)^2}{4}\int_{B_1}\frac{\varphi^2}{|x|^2}\ge 
 \frac{1}{4}\int_{B_1}\frac{\varphi^2}{|x|^2(\log|x|)^2}\,dx.
\end{equation}
Moreover, for any $\varepsilon>0$, there exist a sequence $r_n \downarrow 0$ and a sequence of radial functions $\{\varphi_n\}_{n\in \mathbb{N}}\subset C^{0,1}_{\mathrm{c}}(B_1)$ such that $\varphi_n=0$ on $[0,r_{n+1}] \cup [r_n, 1]$ and 
\begin{equation}
\label{reversecritical}
    \frac{1+\varepsilon}{4}\int_{B_1}\frac{\varphi^2_{n}}{|x|^2(\log|x|)^2}\,dx > \mathcal{I}(\varphi)\hspace{4mm} \text{for any $n\in \mathbb{N}$}.
\end{equation}
\end{proposition}
We remark that the inequality \eqref{criticalhardy} is obtained in \cite{ST}. Moreover, when $N=2$, \eqref{reversecritical} is proved in \cite{kuma2025-2}.

\begin{proof}
For $\varphi\in C^{0,1}_{\mathrm{c}}(B_1)$, we define $\xi:=(-\log r)^{-\frac{1}{2}}\eta$ with $\eta=r^{\frac{N-2}{2}}\varphi$ and $r:=|x|$. Then, we verify that $\eta(0)=\xi(0)=0$ and 
\begin{align*}
r^{N-1}|\nabla\varphi|^2&=r\left|\nabla \eta-\frac{N-2}{2}r^{-1}\eta\frac{x}{r}\right|^2\\
&=r|\nabla \eta|^2+\frac{(N-2)^2}{4r}\eta^2-\frac{N-2}{2}\nabla(\eta^2)\cdot \frac{x}{r}\\
&=r|\nabla \eta|^2+r^{N-1}\frac{(N-2)^2}{4r^2}\varphi^2-\frac{N-2}{2}\nabla(\eta^2)\cdot \frac{x}{r}.
\end{align*}
Moreover, we can confirm that
\begin{align*}
    r|\nabla \eta|^2&=r\left|(-\log r)^{\frac{1}{2}}\nabla\xi-\frac{1}{2}(-\log r)^{-\frac{1}{2}}\xi \frac{x}{r^2}\right|^2\\
    &=r(-\log r)|\nabla\xi|^2+
    \frac{\xi^2}{4r(-\log r)}-\nabla (\xi^2)\cdot\frac{x}{2r}\\
    &=r(-\log r)|\nabla\xi|^2+r^{N-1}\frac{\varphi^2}{4r^2 (\log r)^2}- \nabla (\xi^2)\cdot\frac{x}{2r}.
\end{align*}
Therefore, for any $0\le \varepsilon<1$, we obtain 
\begin{align*}
&\mathcal{I}(\varphi)-\frac{1+\varepsilon}{4}\int_{B_1}\frac{\varphi^2}{|x|^2(\log|x|)^2}\,dx\\
&=\int_{B_1}|x|^{2-N}(-\log |x|)|\nabla\xi|^2\,dx
-\int_{B_1}\frac{\varepsilon\varphi^2}{4|x|^2 (\log|x|)^2}\,dx\\
&\hspace{10mm}-\frac{N-2}{2}\int_{B_1}\nabla (\eta^2)\cdot \frac{x}{|x|^N}\,dx-\frac{1}{2}\int_{B_1}\nabla (\xi^2)\cdot \frac{x}{|x|^N}\,dx\\
&=\int_{B_1}|x|^{2-N}(-\log |x|)|\nabla\xi|^2\,dx
-\int_{B_1}\frac{\varepsilon\varphi^2}{4|x|^2 (\log|x|)^2}\,dx.
\end{align*}
Thus, \eqref{criticalhardy} follows by taking $\varepsilon=0$. In addition, for $0<\varepsilon<1$, we define 
\begin{equation*}
    \xi_{n}(r)=\chi_{[r_{n+1},r_{n}]}(r)\sin t\hspace{4mm}\text{with}\hspace{2mm}
t=\frac{\varepsilon}{2}\log(-\log r)\hspace{4mm}\text{and}\hspace{4mm}r_{n}=e^{-e^{2\pi n/\varepsilon}}.
\end{equation*}
We remark that these functions are also used in \cite[Proposition 2.5]{kuma2025-2} in order to prove \eqref{reversecritical} for the case $N=2$. Then,
we can deduce that $\varphi_n:=r^{\frac{2-N}{2}}(-\log r)^{\frac{1}{2}}\xi_n \in C^{0,1}_{\mathrm{c}}(B_1)$ satisfies $\varphi_n=0$ on $[0,r_{n+1}]\cup [r_n,1]$ and
\begin{align*}
\mathcal{I}(\varphi_n)&-\frac{1+\varepsilon}{4}\int_{B_1}\frac{\varphi_{n}^2}{|x|^2(-\log|x|)^2}\,dx\\
   &=\int_{B_1}|x|^{2-N}(-\log |x|)|\nabla\xi_{n}|^2\,dx
-\int_{B_1}\frac{\varepsilon\varphi_{n}^2}{4|x|^2 (-\log|x|)^2}\,dx\\
&=\frac{1}{2}\varepsilon N\omega_{N}\int_{n\pi}^{(n+1)\pi}\cos^{2} t\,dt - \frac{1}{2}N\omega_{N}\int_{n\pi}^{(n+1)\pi}\sin^{2} t\,dt<0.
\end{align*}
Thus, we get the result.
\end{proof}

\begin{proof}[Proof of Theorem \ref{singularthm}]
We define $\sqrt{\lambda_{*}}$ by the first zero of $V$ and $U_{*}(r)=V(\sqrt{\lambda_{*}}r)$. Then, it follows from Theorem \ref{senthm} that
$(\lambda_{*}, U_{*})$ is the unique radial singular solution of \eqref{gelfand} such that $\lambda(\alpha)\to\lambda_{*}$ and $u(r,\alpha)\to U_{*}$ in $C^2_{\mathrm{loc}}(0,1]$ as $\alpha\to\infty$. From Proposition \ref{h1prop}, we can deduce that $U_{*}\in H^1_{0}(B_1)$.
In addition, thanks to Theorem \ref{asythm}, we obtain
\begin{align*}
    \lambda_{*} f'(U_{*}(r))&=\lambda_{*}f'(V(\lambda_{*}^{1/2}r))\\
    &=\frac{(N-2)^2}{4r^2}+\frac{\gamma\sqrt{N-1}+o(1)}{2^{k-2}r^2(-\frac{\log \lambda_{*}}{2}-\log r)^k}\\
    &=\frac{(N-2)^2}{4r^2}+\frac{\gamma\sqrt{N-1}+o(1)}{2^{k-2}r^2(-\log r)^k} \hspace{4mm}\text{as $r\to 0$.}
\end{align*}
Therefore, thanks to Proposition \ref{criticalprop} and a density argument, we can deduce that $m(U_{*})=\infty$ if (A) is satisfied; while $m(U_{*})<\infty$ if (B) is satisfied.
\end{proof}
By using a similar argument, we obtain 
\begin{corollary}
\label{usecor}
We assume that $f$ satisfies \eqref{asf0}, \eqref{asf1} and \eqref{asF1}. We assume in addition that $(A)$ or $(B)$ is satisfied, where (A) and (B) are those in Theorem \ref{mainthms}.
Then, when $(A)$ is satisfied, there exists a sequence $r_n \downarrow 0$ such that $V$ is unstable in $B_{r_{n}}\setminus B_{r_{n+1}}$ for all $n\in \mathbb{N}$. In addition, when $(B)$ is satisfied, we have $m(V)<\infty$. In particular, $V$ is stable in $B_{r_0}$ with some $r_0>0$.
\end{corollary}
\section{Separation/intersection results}
In this section, we introduce separation and intersection lemmata, which play key roles in clarifying the bifurcation structure. We first introduce the following separation result.
\begin{lemma}
\label{separationlem}
We assume that $N\ge 10$ and 
$f$ satisfies \eqref{asf0} and \eqref{asf1}. Let $v(r,\alpha)$ be a solution of \eqref{gv} with $\alpha>0$ and $V$ be the radial singular solution of \eqref{singulareq}. We also assume that $V$ is stable in $[0,r_{0}]$ with some $r_0>0$ satisfying 
$u_0<V(r_0)-1$, where $u_0$ is that in Lemma \ref{convexlem}. Then, there exists $\alpha_0>0$ such that $v(r,\alpha)<v(r,\beta)<V(r)$ in $(0,r_{0})$ for any $\alpha_0<\alpha<\beta$.
\end{lemma}
\begin{proof}
Thanks to Lemma \ref{basiclem} and Theorem \ref{senthm}, we can deduce that there exists $\alpha_0>0$ such that $v(r,\alpha)>u_0$ for any $r\le r_0$ and $\alpha>\alpha_0$. 

We first suppose that there exist $0<r_1\le r_0$ and $\alpha>\alpha_0$ such that $v(r,\alpha)<V(r,\alpha)$ on $[0,r_1)$ and $v(r_1,\alpha)=V(r_1)$, and derive a contradiction.
We define $W(r):=V(r_1 r)-V(r_1)$. Then, it follows from Lemma \ref{basiclem} that
$W\in C^2_{\mathrm{loc}}(0,1]$ is a stable singular solution of 
\begin{equation}
\label{gelfandw}
\left\{
\begin{alignedat}{4}
 -\Delta w&=\mu f(w+V(r_1))&\hspace{2mm} &\text{in } B_1,\\
w&>0  & &\text{in } \partial B_1, \\
w&=0  & &\text{on } \partial B_1,
\end{alignedat}
\right.
\end{equation}
with $\mu=\mu^{*}:=r_{1}^2$. In addition, $w=v(r_1 r,\alpha)-V(r_1)$ is a solution of \eqref{gelfandw} with $\mu=\mu^{*}$. Thanks to Proposition \ref{h1prop}, we can deduce that $W\in H^1_{0}(B_1)$. 
Moreover, Lemma \ref{convexlem} implies that $g(w):=f(w+V(r_1))$ is non-decreasing and convex. Therefore, it follows from \cite[Theorem 3.1]{Br} that the bifurcation diagram of \eqref{gelfandw} is of Type II. In particular, the equation has the unique weak-solution for $\mu=\mu^{*}$, which is a contradiction. 

Next, we suppose that there exist $0<r_1\le r_0$ and $\alpha_0<\alpha<\beta$ such that $v(r,\alpha)<v(r,\beta)$ on $[0,r_1)$ and $v(r_1,\alpha)=v(r_1,\beta)$, and derive a contradiction. 
Then, $w_1(r):=v(r,\beta)-v(r,\alpha)$ satisfies $w'_{1}(r_1)<0$ and 
the following 
\begin{equation*}
    -\Delta w_1= f(v(r,\beta))-f(v(r,\alpha))\le f'(v(r,\beta))w_1.
\end{equation*}
Here, we use the fact that $v(r,\alpha), v(r,\beta)>u_0$ in $(0,r_0]$ and $f$ is non-decreasing and convex on $u>u_0$, both of which are satisfied by using Lemma \ref{basiclem} and Lemma \ref{convexlem} respectively.
On the other hand, by using Lemma \ref{basiclem} again,
we can deduce that $w(r)=v(r_0r,\beta)-v(r_0,\beta)$ satisfies 
\begin{equation}
\label{gelfandww}
\left\{
\begin{alignedat}{4}
 -\Delta w&=\mu f(w+v(r_0,\beta))&\hspace{2mm} &\text{in } B_1,\\
w&>0  & &\text{in } \partial B_1, \\
w&=0  & &\text{on } \partial B_1
\end{alignedat}
\right.
\end{equation}
with $\mu=r_{0}^2$.
Moreover, by the fact that $v(r,\beta)<V(r)$ on $[0,r_0]$ and $V(r)-v(r,\beta)\to\infty$ as $r\to 0$, there exists $\varepsilon>0$ such that $v(r,\beta)+\varepsilon<V$ in $B_{r_0}$. Thus, it follows from the fact that $u_0<v(r,\beta)$ in $(0,r_0]$ and Lemma \ref{convexlem} that 
\begin{equation}
\label{implicit}
r_0^{2}f'(V(r_0 r))-r_{0}^2 f'(w+v(r_0,\beta))\ge r_{0}^2\varepsilon C_{1}^{-1} V(r_0 r)^{p_0-2} \hspace{4mm}\text{in $B_1$},
\end{equation}
where $C_1$ and $p_0<2$ are those in Lemma \ref{convexlem}. Hence, since $V$ is stable in $B_{r_0}$, $w$ is stable in $B_1$. Moreover, we claim that the first eigenvalue of $-\Delta-r_{0}^2 f'(w+v(r_0,\beta))$ is positive. Indeed,
if the first eigenvalue is equal to $0$, there exists $\varphi\in C^{2}_{0}(\overline{B_1})\setminus\{0\}$ so that 
\begin{equation*}
    \int_{B_1}|\nabla \varphi|^2-r^{2}_{0}f'(w+v(r_0,\beta))\varphi^2\,dx =0.
\end{equation*}
On the other hand, it follows from the stability of $w$ and 
\eqref{implicit} that 
\begin{align*}
\int_{B_1}r^{2}_{0}f'(w+v(r_0,\beta))\varphi^2\,dx &+ \int_{B_1} r_{0}^2\varepsilon C_{1}^{-1} V(r_0 r)^{p_0-2}\varphi^2\,dx \\
 &\le 
 \int_{B_1}r_0^{2}f'(V(r_0 r))\varphi^2\,dx\\
    &\le\int_{B_1}|\nabla \varphi|^2\,dx,
\end{align*}
which is a contradiction.

Hence, it follows from the implicit function theorem that there exist $\delta>0$ and a solution $\hat{w}\in C^2_{0}(\overline{B_{1}})$ of \eqref{gelfandww} for $\mu=(1+\delta)r_{0}^2$ such that $w<\hat{w}$ in $B_1$. Here, we define $w_2:=\hat{w}(r_{0}^{-1}r)-w(r_{0}^{-1}r)$. Then, $w_2$ is a solution of 
\begin{equation*}
    -\Delta w_2=(1+\delta)f(\hat{w}(r_{0}^{-1}r)+v(r_0,\beta))-f(w(r_{0}^{-1}r)+v(r_0,\beta))\ge f'(v(r,\beta))w_2
\end{equation*}
satisfying $w_2>0$ on $(0,r_1]$. Therefore, by the Green's identity, we obtain
\begin{align*}
N\omega_{N}r_1^{N-1}w_1'(r_1)w_2(r_1)&=\int_{B_{r_1}}w_2\Delta w_1-w_1\Delta w_2\ge 0,
\end{align*}
where $\omega$ is the volume of $B_1$. It contradicts the fact that $w'_{1}(r_1)<0$ and $w_2(r_1)>0$.
\end{proof}
Next, we introduce a useful intersection result.
\begin{lemma}
\label{interlem}
We assume that $f$ satisfies \eqref{asf0} and \eqref{asf1}. Let $v(r,\alpha)$ be a solution of \eqref{gv} with $\alpha>0$ and $V$ be the radial singular solution of \eqref{singulareq}. We also assume that $V$ is unstable in $B_{r_2}\setminus B_{r_1}$ with some $0<r_1<r_2$ satisfying $V(r_2)-1>u_0$, where $u_0$ is that in Lemma \ref{convexlem}. Then, there exists $\alpha_0>0$ depending only on $r_1$, $r_2$ and $f$ such that for any $\alpha>\alpha_0$, $v(r,\alpha)$ and $V$ intersects at least one time
on $[r_1,r_2]$. 
\end{lemma}
\begin{proof}
Since $V$ is the radial singular solution and unstable in $B_{r_2}\setminus  B_{r_1}$, there exists $\varphi\in C^{0,1}_{\mathrm{c}}(B_{r_2}\setminus  B_{r_1})$ so that 
\begin{equation*}
    \int_{B_{r_2}\setminus  B_{r_1}}|\nabla \varphi|^2-f'(V)\varphi^2<0.
\end{equation*}
In addition, thanks to Theorem \ref{senthm} and Lemma \ref{basiclem}, there exists $\alpha_0>0$ depending only on $r_1$, $r_2$ and $f$ such that 
\begin{equation*}
    \int_{B_{r_2}\setminus  B_{r_1}}|\nabla \varphi|^2-f'(v(r,\alpha))\varphi^2<0 \hspace{4mm}\text{for any $\alpha>\alpha_0$}
\end{equation*}
and $V(r)>u_0$, $v(r,\alpha)>u_0$ for any $r\le r_2$ and $\alpha>\alpha_0$. 

We assume by contradiction that the sign of $w:=v(r,\alpha)-V$ does not change in $B_{r_2}\setminus B_{r_1}$. Without loss of generality, we assume that $w>0$ in $B_{r_2}\setminus B_{r_1}$. Then, it follows from the fact that $V(r)>u_0$ in $r\le r_2$ and Lemma \ref{convexlem} that
\begin{equation*}
    -\Delta w=f(v(r,\alpha))-f(V)\ge f'(V)w.
\end{equation*}
Thus, we have
\begin{align*}
    \int_{B_{r_2}\setminus B_{r_1}}f'(V)\varphi^2\,dx&\le   \int_{B_{r_2}\setminus B_{r_1}} \nabla w\cdot \nabla(\varphi^2/w)\,dx\\
    &=\int_{B_{r_2}\setminus B_{r_1}}-\frac{|\nabla w|^2}{w^2}\varphi^2+2\frac{\varphi}{w}\nabla\varphi\cdot\nabla w\,dx\\
    &\le \int_{B_{r_2}\setminus B_{r_1}} |\nabla \varphi|^2\,dx,
\end{align*}
which is a contradiction.
\end{proof}

\section{Bifurcation structure}
In this section, we prove Theorem \ref{mainthms}. As a corollary, we obtain Theorem \ref{main1} and Corollary \ref{maincor}. We first prove the following
\begin{theorem}
\label{Ithm}
We assume that \eqref{asf0}, \eqref{asf1}, \eqref{asF1} and (A) are satisfied, where (A) is that in Theorem \ref{mainthms}. Then, the bifurcation diagram is of Type I.
\end{theorem}

\begin{proof}
Thanks to Theorem \ref{singularthm}, it is sufficient to prove that $\lambda(\alpha)$ oscillates infinitely many times around $\lambda_{*}$, where $\lambda(\alpha)$ and $\lambda_{*}$ are those in Theorem \ref{singularthm}. Let $v(r,\alpha)$ be the solution of \eqref{gv} and $V$ be the radial singular solution of \eqref{singulareq}. We remark that $v(\sqrt{\lambda(\alpha)},\alpha)=0$ and $V(\sqrt{\lambda_{*}})=0$. Then, thanks to Lemma \ref{basiclem}, Lemma \ref{interlem} and Corollary \ref{usecor}, the intersection number of $v(r,\alpha)$ and $V$
on $[0,\min\{\sqrt{\lambda(\alpha)},\sqrt{\lambda}_{*}\}]$ converges to infinity as $\alpha\to\infty$. It implies that $\lambda(\alpha)$ intersects with $\lambda^{*}$ infinitely many times as $\alpha\to\infty$ (this assertion can be proved rigorously by using an argument similar to that of \cite[Lemma 5]{Mi2015}).
Thus, we obtain the conclusion.
\end{proof}
Next, we prove the following
\begin{proposition}
\label{mainprop}
We assume that $f$ satisfies \eqref{asf0}, \eqref{asf1}, \eqref{asF1} and (B), where (B) is that in Theorem \ref{mainthms}.
Then, there exists $\alpha_0>0$ such that if $(\lambda(\alpha), u(r,\alpha))$ is a solution of \eqref{gelfand} satisfying $\alpha>\alpha_{0}$, we have $\lambda'(\alpha)\neq 0$.    
\end{proposition}
\begin{proof}
Let us denote by $\cdot$ the differentiation with respect to $\alpha$. Let $i\in \mathbb{N}$ be a number depending only on $f$ to be defined later. Then, the conclusion is obtained if the following assertion holds: there exists $\alpha_0>0$ such that if $\alpha_0<\alpha$ and $\dot{\lambda}(\alpha)=0$ are satisfied, then $(-1)^{i-1}\ddot{\lambda}(\alpha)>0$. We assume by contradiction that there exist a sequence of solutions $(\lambda(\alpha_n), u(r,\alpha_n)))$ of \eqref{gelfand} such that $\dot{\lambda}(\alpha_n)=0$, $(-1)^{i-1}\ddot{\lambda}(\alpha_n)\le 0$ and $\alpha_n\to \infty$ as $n\to\infty$. Now, we denote by 
$u_n(r):=u(r,\alpha_n)$ and $\lambda_n=\lambda(\alpha_n)$. Then, by a direct calculation, we verify that 
\begin{equation}
\label{equdot}
\left\{
\begin{alignedat}{4}
 &-\Delta \dot{u}_{n}= \lambda_n f'(u_{n})\dot{u}_n,\hspace{14mm}0<r\le 1,\\
&\dot{u}_{n}(0)=1, \hspace{1mm} \frac{d}{dr} \dot{u}_{n}(0)=0, \hspace{1mm} \dot{u}_{n}(1)=0
\end{alignedat}
\right.
\end{equation}
and
\begin{equation}
\label{equdots}
\left\{
\begin{alignedat}{4}
-\Delta \ddot{u}_{n}&=\lambda_n f'(u_n)\ddot{u}_{n}+
\lambda_n  f''(u_n) \dot{u}^{2}_{n}+
\ddot{\lambda}_n f(u_n)
\hspace{4mm}&&\text{in $B_1$},\\
\ddot{u}_{n}&=0 &&\text{on $\partial B_1$}.\notag
\end{alignedat}
\right.
\end{equation}
Here, we set $v(r,\alpha):= u(\lambda(\alpha)^{-1/2}r,\alpha)$ and $v_n:=v(r,\alpha_n)$. Then, $v$ is a solution of \eqref{gv}. Therefore, thanks to Theorem \ref{senthm}, we can deduce that $\lim_{n\to \infty}\lambda_n=\lambda_{*}$ with some $\lambda_{*}>0$ and $v_{n}\to V$ in $C^2_{\mathrm{loc}}(0,\infty)$, where $V$ is the radial singular solution of \eqref{singulareq}. In addition, by Lemma \ref{basiclem}, Corollary \ref{usecor} and Lemma \ref{separationlem}, we have $\dot{v}_{n}\ge 0$ and $v_n\le V$ for all $0<r<r_{0}$ and $n$ sufficiently large
with some $r_0>0$ independent of $n$. Moreover, thanks to Theorem \ref{asythm}, there exist $\delta>0$ and $r_1<r_0$ independent of $n$ such that we have
\begin{equation*}
    f'(v_k)\le f'(V)\le \frac{(N-2)^2}{4}|x|^{-2}+ \frac{1-2\delta}{4}|x\log x|^{-2}
    \hspace{4mm}\text{in $B_{r_1}$}
\end{equation*}
for all $n$ sufficiently large. Thus thanks to the above estimate and Theorem \ref{singularthm},
we obtain
\begin{align}
\label{estimateI}
\begin{split}
    \lambda_n f'(u_n)&\le \frac{(N-2)^2}{4}|x|^{-2}+ \frac{1-2\delta}{4}|x(\log x+\frac{\log \lambda_n}{2})|^{-2}+C\\
    &\le \frac{(N-2)^2}{4}|x|^{-2}+\frac{1-\delta}{4}|x(\log x)|^{-2}+C
    \end{split}
\end{align}
for all $0<r<1$ and $n$ sufficiently large, where $C$ is a positive constant independent of $n$. Moreover, since $\dot{v}_n=\dot{u}_n(\lambda_n^{-1/2}r,\alpha_n)$, it follows that $\dot{u}_n\ge 0$ for $0<r<\frac{1}{2}\lambda_{*}^{-1/2}r_0$ for all $n$ sufficiently large. Now, let $\varepsilon>0$ be a small constant defined later. Then, Theorem \ref{asythm} implies that 
\begin{equation}
\label{lower-es}
  \lambda_{*} f'(U_*)\ge \frac{(N-2)^2-4\varepsilon^2}{4|x|^2} 
\end{equation}
for $0<r<\hat{r}$ with some $0<\hat{r}<\frac{1}{2}\lambda_{*}^{-1/2}r_0$.

We define 
\begin{equation*}
    w_n=\frac{|x|^{\frac{N-2}{2}} \dot{u}_{n}}{\lVert |x|^{\frac{N-2}{2}} \dot{u}_{n} \rVert_{L^2(B^{2}_{1})}},
\end{equation*}
where $B^{2}_{1}$ is the $2$-dimensional unit ball. Then, $w_n$ satisfies 
\begin{equation*}
\left\{
\begin{alignedat}{4}
 -\Delta w_{n}&= \left(\lambda_nf'(u_n)-\frac{(N-2)^2}{4|x|^2}\right) w_{n}, \hspace{4mm}&&\text{in $B^{2}_1$},\\
w_{n}&=0 &&\text{on $\partial B^{2}_1$}
\end{alignedat}
\right.
\end{equation*}
and $w_{n}(r)\ge 0$ in $0<r<\hat{r}$. Hence, it follows from \eqref{estimateI}, Proposition \ref{criticalprop} and a density argument that
\begin{align*}
  \delta\int_{B^{2}_{1}} |\nabla w_n|^2\,dx&\le\int_{B^{2}_{1}} |\nabla w_n|^2\,dx- \frac{1-\delta}{4}\int_{B^{2}_{1}}|x\log x|^{-2} w^2_{n}\,dx\\
    &=\int_{B^{2}_{1}}\left(\lambda_nf'(u_n)-\frac{(N-2)^2}{4|x|^2}-\frac{1-\delta}{4}|x\log x|^{-2}\right) w^2_{n}\,dx\\
&\le C\int_{B^{2}_1} w^2_{n}\,dx=C.
\end{align*}
As a result, there exists $w\in H^{1}_{0}(B^2_{1})$ such that $w_n\rightharpoonup w$ in $H^{1}(B^{2}_1)$ by taking a subsequence if necessary. In particular, we get $w_n\to w $ in $L^q(B^{2}_1)$ for any $1\le q<\infty$. Thus, we have $\lVert w\rVert_{L^2(B^2_{1})}=1$. Moreover, by Theorem \ref{singularthm}, the elliptic regularity theory (see \cite{gil}) and a diagonal argument, we can deduce that $w_n\to w$ in $C^2_\mathrm{loc}(0,1]$ as $n\to\infty$ and $w\in C^2_{\mathrm{loc}}(0,1]$ is a radial function satisfying $w\ge 0$ for $0<r<\hat{r}$ and
\begin{equation}
\label{weq}
\left\{
\begin{alignedat}{4}
 &w''+\frac{w'}{r}+ \left(\lambda_{*}f'(U_{*})-\frac{(N-2)^2}{4r^2}\right) w=0, \hspace{4mm}\text{$0<r<1$},\\
&w(1)=0.
\end{alignedat}
\right.
\end{equation}
If the equation \eqref{weq} does not have a non-trivial solution $\hat{w}\in C^2(0,1]$ satisfying $\hat{w}\ge 0$ in $(0,t)$ with some $t>0$, we obtain a contradiction. On the other hand, if \eqref{weq} has a non-trivial solution $\hat{w}\in C^2(0,1]$ satisfying $\hat{w}\ge0$ in $(0,t)$ with some $t>0$, 
we verify that $\hat{w}'(1)w'(1)>0$. Indeed, if $w(s)=0$ with $s\in (0,1)$, then it follows from Green's identity that
\begin{equation*}
    2\pi s (w'(s)\hat{w}(s)-w(s)\hat{w}'(s))=0.
\end{equation*}
By Hopf's lemma, we have $w'(s)\neq 0$. Therefore, we have $\hat{w}(s)=0$. By using a similar argument again, we verify that $w(s)=0$ if and only if $\hat{w}(s)=0$. Hence, since $w$ and $\hat{w}$ are non-trivial solution satisfying $w\ge 0$ and $\hat{w}\ge 0$ in $(0,\min\{\hat{r},t\})$, it follows from Hopf's lemma that $w'(1)$ and $\hat{w}'(1)$ has the same sign. Therefore, the result follows. We choose a non-trivial solution $\hat{w}$ satisfying $\hat{w}>0$ in $(0,t)$ with some $t>0$, and define $i\in \{1,2\}$ so that $(-1)^{i-1}\hat{w}'(1)>0$. By the above discussion, $i$ is defined 
regardless of the choice of $\hat{w}$. Moreover, we have $(-1)^{i-1}w'(1)>0$.

We define $z:=r^{-\varepsilon}w$. We can verify that $z\ge 0$ for $0<r<\hat{r}$ and $z$ satisfies
\begin{equation*}
    z''+\frac{1+2\varepsilon}{r}z'+ \left(\lambda_{*}f'(U_{*})-\frac{(N-2)^2-4\varepsilon^2}{4r^2}\right) z=0.
\end{equation*}
Here, we claim that $z'\le 0$ for $0<r<\hat{r}$. Indeed, if not, there exists $0<t<\hat{r}$ such that $z'(t)>c_1$ with some $c_1>0$. From the above equation and \eqref{lower-es}, we have
$(r^{1+2\varepsilon}z')'\le 0$ for all $0<r<\hat{r}$. Hence, $z'(s)\ge s^{-(1+2\varepsilon)}t^{1+2\varepsilon}z'(t)\ge s^{-(1+2\varepsilon)}c_2$ for all $0<s<t$, where $c_2=t^{1+2\varepsilon}c_1$. Integrating the inequality on $(r,t)$, we have $-z(r)\ge -z(t)+\frac{c_2}{2\varepsilon}(r^{-2\varepsilon}-t^{-2\varepsilon})\to\infty$ as $r\to 0$, which contradicts the fact that $z\ge 0$ for $0<r<\hat{r}$.  

Hence, since $\lVert w\rVert_{L^2(B_{1}^2)}=1$, we obtain $z>z_0$ for $0<r<\frac{\hat{r}}{2}$ with some $z_0>0$. Thanks to \eqref{equdot} and \eqref{equdots}, we have
\begin{equation*}
    0=\int_{B_1}\ddot{\lambda}_n f(u_n)\dot{u}_n\,dx+\int_{B_1}\lambda_n f''(u_n)\dot{u}^{3}_n\,dx=\ddot{\lambda}_n\mathcal{A}_n+ \mathcal{B}_n,
\end{equation*}
where
\begin{align*}
\mathcal{A}_n:= \int_{B_{1}}f(u_n)\dot{u}_n\,dx \hspace{4mm}\text{and}\hspace{4mm}\mathcal{B}_n:=\int_{B_1}\lambda_n f''(u_n)\dot{u}^{3}_n\,dx.
\end{align*}
We claim that there exists $n_0$ such that 
$(-1)^{i}A_n>0$, $B_n>0$ for all $n>n_0$, which contradicts the equality above.
Indeed, we recall that $w_n\to w $ in $L^q(B^{2}_1)$ for any $1\le q<\infty$ and $v_n\to V_{*}$ in $C^2_{\mathrm{loc}}(0,1]$. Therefore, by Theorem \ref{singularthm} and Fatou's Lemma, we get
\begin{align*}
\liminf_{n\to\infty}\frac{\mathcal{B}_n}{\lVert |x|^{\frac{N-2}{2}}\dot{u}_n\rVert_{L^2(B^2_{1})}^3}&=\liminf_{n\to\infty}\int_{B_{1}}\lambda_n |x|^{\frac{3(2-N)}{2}}f''(u_n)w^3_{n}\,dx\\
    &= \lambda_{*} \int_{B_{\frac{\hat{r}}{2}}} |x|^{\frac{3(2-N)}{2}}f''(U_{*})w^3\,dx-C\\
    &\ge \lambda_{*}\int_{B_{\frac{\hat{r}}{2}}} |x|^{\frac{3(2-N)}{2}+3\varepsilon}f''(U_{*})z_{0}^3\,dx-C\\
    &=\infty \hspace{4mm}\text{by taking $\varepsilon$ such that $0<\varepsilon<\frac{2}{3}$.}
\end{align*}
Moreover, if follows from \cite[Theorem 1.5]{korbook} that
\begin{equation*}
    \int_{B_1}f(u_n)\dot{u}_n=\frac{N\omega_{N}}{2\lambda_n}u'_{n}(1)\dot{u}_{n}'(1),
\end{equation*}
where $\omega_{N}$ is the volume of $B_1$. Hence, we have
\begin{align*}
\lim_{n\to\infty}\frac{\mathcal{A}_n}{\lVert |x|^{\frac{N-2}{2}}\dot{u}_n\rVert_{L^2(B^2_{1})}}&=\lim_{n\to\infty}\frac{1}{\lVert |x|^{\frac{N-2}{2}}\dot{u}_n\rVert_{L^2(B^2_{1})}}\int_{B_{1}} f(u_n)\dot{u}_n\,dx\\
&=\lim_{n\to\infty} \frac{N\omega_N u'_{n}(1)\dot{u}_{n}'(1)}{2\lambda_n{\lVert |x|^{\frac{N-2}{2}}\dot{u}_n\rVert_{L^2(B^2_{1})}}}\\
&=\frac{N\omega_{N}U'_{*}(1)}{2\lambda_{*}}\lim_{n\to\infty} (w_{n}'(1)-\frac{N-2}{2}w_{n}(1))\\
&=\frac{N\omega_{N}U'_{*}(1)w'(1)}{2\lambda_{*}}=(-1)^{i}c
\end{align*}
with some positive constant $c>0$.
Thus, we get the claim.  
\end{proof}

\begin{proof}[Proof of Theorem \ref{main1}, Theorem \ref{mainthms} and Corollary \ref{maincor}.]
We first claim that Theorem \ref{main1} and Corollary \ref{maincor} are obtained by using Theorem \ref{mainthms} and Lemma \ref{examplelem}. Thus, it suffices to prove Theorem \ref{mainthms}. We also remark that Theorem \ref{mainthms} (i) is already proved as Theorem \ref{Ithm}. 

Now, we prove Theorem \ref{mainthms} (ii). Thanks to Theorem \ref{singularthm} and Proposition \ref{criticalprop}, we verify that there exists $r_0$ such that the radial singular solution $U_{*}$ of \eqref{gelfand} with $\lambda=\lambda_{*}$ is stable in $B_{r_0}$ and $U_{*}(r_0)>u_0+1$, where $u_0>0$ is that in Lemma \ref{convexlem}. Let $c\ge U_{*}(r_0)$ and we denote by $r_1$ the first zero of $U_{*}-c$. We define $\hat{U}_{*}(r):=U_{*}(r_1r)-c$ and $\hat{\lambda}_{*}=\lambda_{*}r_{1}^2$.
Then, thanks to Lemma \ref{basiclem} and Lemma \ref{convexlem}, we can deduce that $(\hat{\lambda}_{*}, \hat{U}_{*})$ is the radial stable singular solution of \eqref{gelfand} with the non-decreasing and convex nonlinearity $f_{c}(u):=f(c+u)$. Moreover, thanks to Proposition \ref{h1prop}, we can deduce that $\hat{U}_{*}\in H^1_{0}(B_1)$.
Hence, it follows from \cite[Theorem 3.1]{Br} that the bifurcation diagram is of Type II with $f_{c}$. In addition, thanks to Proposition \ref{mainprop}, we can deduce that $\lambda(\alpha)$ is non-degenerate for all $\alpha$ sufficiently large. Hence, the result is satisfied.

Finally we prove Theorem \ref{mainthms} (iii). In this case, since $f$ is analytic on $(-\delta,\infty)$, by the analytic implicit functional theorem, we can deduce that $\lambda(\alpha)$ is analytic, which implies that the set of zeros of $\lambda'$ does not have an accumulation point. Thus, the desired result follows.
\end{proof}

\subsection*{Acknowledgment}
The author would like to thank Professor Michiaki Onodera for his valuable advice to greatly improve the clarity of the presentation. 

\appendix
\section{A similarity between Type II and Type III}
In this appendix, we prove the following
\begin{proposition}
\label{apenprop}
We assume that $f$ satisfies \eqref{asf0} and \eqref{asf1}. We assume in addition that  
$f$ can be analytically extended to $(-\delta,\infty)$ with some $\delta>0$. Moreover, we suppose that the bifurcation diagram is of Type III with $f$. Then, there exists $c_0>0$ such that the bifurcation diagram is of Type II with $f_{c}=f(u+c)$ for all $c>c_0$.
\end{proposition}
\begin{proof}
Since $f$ is analytic and $\lambda$ is parameterized by $\alpha$, we can deduce that the bifurcation curve is analytic, which has no self-intersection and no secondary bifurcation point.
Hence, by using the method of the proof of \cite[Theorem 1]{Dan} and the analytic bifurcation theory provided in \cite[Section 2.1]{BD}, we can verify that the $m(u,\alpha)$ changes at $\alpha$ if $\lambda(\alpha)$ turns at $\alpha$. Thus, we have $m(u(r,\alpha))\le M$ with some $M\in\mathbb{N}$ since the bifurcation diagram is of Type III. 

Here, we claim that $m(U_{*})\le M$. Indeed, if $m(U_{*})\ge M+1$, there exists a $M+1$ dimensional subspace $X_{M+1}\subset C^{1}_{\mathrm{c}}
(B_1)$ so that $Q_{U_{*}}(\xi)<0$ for all $\xi\in X_{M+1}\setminus\{0\}$. Here, we mention that since $m(u(r,\alpha))\le M$, we get a contradiction if we obtain the following assertion: $Q_{u(r,\alpha)}(\xi)<0$ for all $\xi\in X_{M+1}\setminus\{0\}$ and $\alpha>0$ sufficiently large. In order to obtain the assertion, we suppose to the contrary, i.e., there exist a sequence $\alpha_n$ and $\xi_n \in X_{M+1}\setminus\{0\}$ with $\lVert \xi_n \rVert_{C^1} =1$
such that $\alpha_n\to\infty$ and $Q_{u(r,\alpha_n)}(\xi_n)\ge 0$ as $u\to\infty$. Since $X_{M+1}\subset C_{\mathrm{c}}^{1}(B_1)$ is finite dimensional, we can deduce that there exists $\xi\in X_{M+1}$ with $\lVert \xi \rVert_{C^1} =1$ such that $\xi_n\to \xi$ in $C^{1}(B_1)$.
Therefore, thanks to Lemma \ref{convexlem}, Theorem \ref{senthm} and Fatou's Lemma, we obtain 
\begin{align*}
    Q_{U_{*}}(\xi)&=\int_{B_1}|\nabla \xi|^2 - \lambda_{*} f'(U_{*})\xi^2\,dx\\&\ge \liminf_{n\to\infty}\int_{B_1}|\nabla \xi_n|^2 - \lambda(\alpha_n) f'(u(r,\alpha_n))\xi_{n}^2\,dx\ge 0,
\end{align*}
which is a contradiction. 

Thus, it follows that $m(U_{*})\le M$. Hence, thanks to \cite[Proposition 1.5.1]{Dup}, there exists $r_0>0$ such that $U_{*}$ is stable in $B_{r_0}$ and $U_{*}(r_0)>u_0+1$, where $u_0$ is that in Lemma \ref{convexlem}. Therefore, by an argument similar to that in the proof of Theorem \ref{mainthms} (ii), we obtain the conclusion. 
\end{proof}

\bibliographystyle{plain}
\bibliography{classification_of_bifurcation}

\end{document}